\documentclass[aos]{imsart}

\RequirePackage{amsthm,amsmath,amsfonts,amssymb}
\RequirePackage[authoryear]{natbib}
\RequirePackage[colorlinks,citecolor=blue,urlcolor=blue]{hyperref}
\RequirePackage{graphicx}
\usepackage{subcaption}
\startlocaldefs
\usepackage{amsmath,amssymb,amsthm,ifthen,verbatim,amsfonts}

\usepackage{calrsfs}
\usepackage{epsf,color,framed,fancybox,shadow,enumerate,appendix,stackrel}
\usepackage[T1]{fontenc}
\usepackage{latexsym}
\usepackage{graphics}
\usepackage{wrapfig}
\usepackage{makeidx}
\usepackage{algorithm2e}
\usepackage{latexsym}

\usepackage{aliascnt}
\usepackage{bbm}

\usepackage{multirow}
\usepackage{framed}
\usepackage{stmaryrd}
\usepackage{xcolor}

\usepackage[shortlabels]{enumitem}

\usepackage{pslatex,fancyhdr}

\usepackage{tikz,pdfsync}
\usepackage{hyperref}
\usepackage{cleveref}

\usepackage{makeidx}

\newcounter{notecounter}
\newcounter{rem}
\newenvironment{remark}{\refstepcounter{rem} \noindent \textbf{Remark \therem.} \small \begin{sf}}
{\end{sf}}
\crefname{rem}{remark}{remark}
\Crefname{rem}{Remark}{Remark}

\newenvironment{hyp}[1]{
    \begin{enumerate}[label=\textbf{\sf(#1\arabic*)},resume=hyp#1]\begin{sf}}{\end{sf}\end{enumerate}}

\newtheorem{lemma}{Lemma}
\newtheorem{theorem}{Theorem}
\newtheorem{proposition}{Proposition}

\theoremstyle{remark}

\endlocaldefs

\begin{document}

\newcommand{\arginf}{\mathrm{arginf}}
\newcommand{\argmin}{\mathrm{argmin}}
\newcommand{\argmax}{\mathrm{argmax}}
\newcommand{\as}{a.s.}
\newcommand{\asconv}[1]{\stackrel{#1-a.s.}{\rightarrow}}
\newcommand{\Aset}{\mathsf{A}}
\newcommand{\ball}[1]{\mathsf{B}(#1)}
\newcommand{\tball}[1]{\mathsf{\bar B}(#1)}
\newcommand{\beps}{{\boldsymbol{\epsilon}}}
\newcommand{\bproof}{\textbf{Proof :}\quad}
\newcommand{\bmuf}[2]{b_{#1,#2}}
\newcommand{\bz}{{\mathbf{z}}}
\newcommand{\card}{\mathrm{card}}
\newcommand{\chunk}[3]{{#1}_{#2:#3}}
\newcommand{\convprob}[1]{\stackrel{#1-\text{prob}}{\rightarrow}}
\newcommand{\convprobdown}[2]{\stackrel[#2]{#1-\text{prob}}{\rightarrow}}
\newcommand{\Cov}{\mathbb{C}\mathrm{ov}}
\newcommand{\CPE}[2]{\PE\lr{#1| #2}}
\newcommand{\Cset}{{\mathsf{C}}}
\renewcommand{\det}{\mathrm{det}}
\newcommand{\dimlabel}{\mathsf{m}}
\newcommand{\dimU}{\mathsf{q}}
\newcommand{\dimX}{\mathsf{d}}
\newcommand{\dimY}{\mathsf{p}}
\newcommand{\dlim}{\Rightarrow}
\newcommand{\e}[1]{{\left\lfloor #1 \right\rfloor}}
\newcommand{\eproof}{\quad \Box}
\newcommand{\eremark}{</WRAP>}
\newcommand{\eqdef}{:=}
\newcommand{\eqlaw}{\stackrel{\mathcal{L}}{=}}
\newcommand{\eqsp}{\;}
\newcommand{\Eset}{ {\mathsf E}}
\newcommand{\Esigma}{ {\mathcal{E}}}
\newcommand{\esssup}{\mathrm{essup}}
\newcommand{\falph}{f}
\newcommand{\fr}[1]{{\left\langle #1 \right\rangle}}
\renewcommand{\geq}{\geqslant}
\newcommand{\hchi}{\hat \chi}
\newcommand{\Id}{\mathrm{Id}}
\newcommand{\img}{\text{Im}}
\newcommand{\indi}[1]{\mathbf{1}_{#1}}
\newcommand{\indiacc}[1]{\mathbf{1}_{\{#1\}}}
\newcommand{\indin}[1]{\mathbf{1}\{#1\}}
\renewcommand{\j}{j}
\newcommand{\klbck}[2]{\mathrm{K}\lr{#1||#2}}
\renewcommand{\k}{k}
\newcommand{\lo}{k}
\renewcommand{\l}{\mathcal{L}}
\newcommand{\law}{\mathcal{L}}
\newcommand{\labelinit}{\pi}
\newcommand{\labelkernel}{Q}
\renewcommand{\leq}{\leqslant}
\newcommand{\lk}{\tilde{\mathcal{L}}}
\newcommand{\lone}{\mathsf{L}_1}
\newcommand{\lrav}[1]{\left|#1 \right|}
\newcommand{\lr}[1]{\left(#1 \right)}
\newcommand{\lrb}[1]{\left[#1 \right]}
\newcommand{\lrc}[1]{\left\{#1 \right\}}
\newcommand{\lrcb}[1]{\left\{#1 \right\}}
\newcommand{\ltwo}[1]{\PE^{1/2}\lrb{\lrcb{#1}^2}}
\newcommand{\Ltwo}{\mathrm{L}^2}
\newcommand{\mc}[1]{\mathcal{#1}}
\newcommand{\mcbb}{\mathcal B}
\newcommand{\mcf}{\mathcal{F}}
\newcommand{\mcg}{\mathcal{G}}
\newcommand{\meas}[1]{\mathrm{M}_{#1}}
\newcommand{\nell}{_n^\k}
\newcommand{\nelln}{_n^{\k_n}}
\newcommand{\norm}[1]{\left \| #1 \right \|}
\newcommand{\norms}[2]{\left\|#1\right\|_{#2}}
\newcommand{\normsx}[3]{\left\|#1\right\|_{\mathsf{L}^{#2}_{#3}}}
\newcommand{\normsxsimp}[2]{\left\|#1\right\|_{\mathsf{L}_{#2}}}
\newcommand{\note}[2]{\refstepcounter{notecounter} \begin{leftbar}  \textcolor{red}{\textbf{#1}} \textcolor{violet}{(\textbf{Note} \thenotecounter)}: \begin{sf}{\color{blue} #2} \end{sf}  \end{leftbar}}
\newcommand{\nset}{\mathbb N}
\newcommand{\one}{\mathsf{1}}
\newcommand{\phiv}{{\phi^\star}}
\newcommand{\pscal}[2]{\langle #1,#2\rangle}
\newcommand{\PE}{\mathbb E}
\newcommand{\PP}{\mathbb P}
\newcommand{\Psif}{\Psi_f}
\newcommand{\psconv}{\stackrel{\PP-a.s.}{\rightarrow}}
\newcommand{\qset}{\mathbb Q}
\newcommand{\rmd}{\mathrm d}
\newcommand{\rme}{\mathrm e}
\newcommand{\rmi}{\mathrm i}
\newcommand{\Rset}{\mathbb{R}}
\newcommand{\rset}{\mathbb{R}}
\newcommand{\rsetpos}{\mathbb{R}^+_*}
\newcommand{\rti}{\sigma}
\newcommand{\seq}[2]{\lrc{#1\eqsp: \eqsp #2}}
\newcommand{\set}[2]{\lrc{#1\eqsp: \eqsp #2}}
\newcommand{\sg}{\mathrm{sgn}}
\newcommand{\supnorm}[1]{\left\|#1\right\|_{\infty}}
\renewcommand{\t}{\tilde}
\newcommand{\tmu}{ {\tilde{\mu}}}
\newcommand{\tphi}{\tilde \phi}
\newcommand{\ttheta}{\tilde \theta}
\newcommand{\thv}{{\theta^\star}}
\newcommand{\Tset}{ {\mathsf{T}}}
\newcommand{\Tsigma}{ {\mathcal{T}}}
\newcommand{\tv}[1]{\left\|#1\right\|_{TV}}
\newcommand{\unif}{\mathrm{Unif}}
\newcommand{\Uset}{{\mathsf U}}
\newcommand{\Usigma}{{\mathcal U}}
\newcommand{\V}{{\mathcal V}}
\newcommand{\Var}{\mathbb{V}\mathrm{ar}}
\newcommand{\Vratio}{\mathbb{V}}
\newcommand{\Vset}{\mathcal{V}}
\newcommand{\vz}{z}
\newcommand{\weakconv}{\stackrel{\mathcal{L}}{\rightsquigarrow}}
\newcommand{\Wset}{\mathcal{W}}
\newcommand{\xivp}[1]{{\xi^\star_{#1}}}
\newcommand{\xiv}{{\xi^\star}}
\newcommand{\x}{\mathbf{x}}
\newcommand{\X}{{\mathfrak{X}}}
\newcommand{\Xset}{{\mathsf X}}
\newcommand{\Xsigma}{\mathcal X}
\newcommand{\Yset}{{\mathsf Y}}
\newcommand{\Ysigma}{\mathcal Y}
\newcommand{\zset}{\mathbb{Z}}
\newcommand{\Zset}{\mathsf{Z}}
\newcommand{\Zsigma}{\mathcal{Z}}

\begin{frontmatter}
\title{On the Asymptotics of Importance Weighted Variational Inference}
\runtitle{Asymptotics of Importance Weighted Variational Inference}
\begin{aug}
\author[A]{\fnms{Badr-Eddine}~\snm{Cherief-Abdellatif}\ead[label=e1]{badr-eddine.cherief-abdellatif@cnrs.fr}},
\author[B]{\fnms{Randal}~\snm{Douc}\ead[label=e2]{randal.douc@telecom-sudparis.eu}}
\author[C]{\fnms{Arnaud}~\snm{Doucet}\ead[label=e3]{doucet@stats.ox.ac.uk}}
\and
\author[B]{\fnms{Hugo}~\snm{Marival}\ead[label=e4]{hugo.marival@telecom-sudparis.eu}}

\address[A]{CNRS, LPSM, Sorbonne University \printead[presep={,\ }]{e1}}

\address[B]{SAMOVAR, Telecom Sudparis, Institut Polytechnique de
  Paris\printead[presep={,\ }]{e2,e4}}
\address[C]{Oxford University\printead[presep={,\ }]{e3}}

\end{aug}

\begin{abstract}
For complex latent variable models, the likelihood function is not available in closed form. In this context, a popular method to perform parameter estimation is Importance Weighted Variational Inference. It essentially maximizes the expectation of the logarithm of an importance sampling estimate of the likelihood with respect to both the latent variable model parameters and the importance distribution parameters, the expectation being itself with respect to the importance samples. Despite its great empirical success in machine learning, a theoretical analysis of the limit properties of the resulting estimates is still lacking. We fill this gap by establishing consistency when both the Monte Carlo and the observed data sample sizes go to infinity simultaneously. We also establish asymptotic normality and efficiency under additional conditions relating the rate of growth between the Monte Carlo and the observed data samples sizes. We distinguish several regimes related to the smoothness of the importance ratio.
\end{abstract}

    
    
\end{frontmatter}

\section{Introduction}
Consider a latent variable models specified by a joint density $p_\theta(x,z)$, where $x$ is an observed variable while $z$ is a latent variable. The (marginal) likelihood function is then computed from the joint distribution via $p_\theta(x)=\int p_\theta(x,z) dz$ and is assumed not to be available in closed form.
Variational Inference (VI) addresses the intractability of likelihood calculations by optimizing a tractable lower bound on the log-likelihood \citep{Jordan1999,blei2017variational}. This paper focuses on Importance Weighted Variational Inference (IWVI) proposed by \cite{IWAE2015}, itself inspired by \cite{bornschein2014reweighted}, which was originally developed to train unsupervised deep generative models like variational autoencoders \citep{kingma2013auto,rezende2014stochastic}.

IWVI creates a tight log-likelihood lower bound by calculating the expectation of an importance sampling likelihood estimate and maximizing this bound with respect to model parameters $\theta$  and the parameter of the importance sampling distribution; the expectation being with respect to the importance samples. 

\subsection{Importance Weighted Variational Inference}

To introduce a more detailed mathematical framework, we consider some data $x \in \Xset$ generated by a latent variable $z \in \Zset$ through a joint model with density $p_\theta(x,z)=p_\theta(x|z)p_\theta(z)$ with respect to some reference measure on $\Xset\times\Zset$. 
Since $z$ is hidden and only $x$ is observed, the log-(marginal) likelihood is defined as
$$
\log p_\theta(x) = \log \Bigl(\int_\mathcal{Z} p_\theta(x,z) dz\Bigr) .
$$

Given any positive unbiased estimate $\hat{p}_\theta(x)$ of the likelihood $p_\theta(x)=\mathbb{E}[\hat{p}(x)]$, VI exploits Jensen's inequality to obtain a so-called Evidence Lower Bound (ELBO) $\mathbb{E}[\log\hat{p}(x)]\leq\log\mathbb{E}[\hat{p}(x)]=\log p(x)$ that is well suited for stochastic optimization. In this paper, we use importance sampling (IS) to define an unbiased estimate $\hat{p}^{(k)}_\theta(x)$ of the (marginal) likelihood $\hat{p}_{\theta}(x)$ of the form
$$
\hat{p}^{(k)}_{\theta}(x) := \frac{1}{k} \sum_{\ell=1}^k\frac{p_\theta(x,z^\ell)}{q_\phi(z^\ell|x)},
$$
where $z^\ell \overset{\textup{i.i.d.}}{\sim}  q_\phi(\cdot|x)$,  $q_\phi(\cdot|x)$ is a so-called importance distribution, $p_\theta(x,z)/q_\phi(z|x)$ the importance ratio and $\phi$ the variational parameter. As $k$ increases, we obtain a low variance estimate of estimate of $p_\theta(x)$, and thus  $\mathbb{E}[\log\hat{p}^{(k)}_{\theta}(x)]$ is a tight lower bound on $\log p_\theta(x)$ that is the objective to maximize. 
Note that even though we focus on IS in this paper, many alternative Monte Carlo approximations have recently been suggested to address challenging problems in machine learning, including sequential models \citep{maddison2017filtering,naesseth2018variational,le2018auto}, missing values \citep{mattei2019miwae,ipsen2020not}, causal inference \citep{mayer2020missdeepcausal}, or general Bayesian models \citep{domke2018importance,domke2019divide}.

Despite its popularity in machine learning, very few theoretical results are available for IWVI, and most of them focus on results quantifying the gap between the log-likelihood and its variational lower bound. It was first shown in \cite{IWAE2015} that the gap is not increasing with $k$, i.e. $\mathbb{E}[\log\hat{p}^{(k)}_\theta(x)]\leq\mathbb{E}[\log\hat{p}^{(k+1)}_\theta(x)]$ while $\mathbb{E}[\log\hat{p}^{(k)}_\theta(x)]\rightarrow\log p_\theta(x)$ in probability as $k\rightarrow\infty$ when the importance weights are bounded. 
Under moments assumptions on the importance weight, \cite{maddison2017filtering,nowozin2018debiasing} later provided expansions of the form 
\begin{equation}
\label{devIWELBO}
\mathbb{E}[\log\hat{p}^{(k)}_\theta(x)] = \log p_\theta(x) - \frac{\mathbb{V}[\hat{p}^{(1)}_\theta(x)]}{2p_\theta(x)^2} \cdot \frac{1}{k} 
+ \mathcal{O}\left(\frac{1}{k^2}\right).
\end{equation}
 Such formulas offer an interesting perspective on the IWVI objective, it is simply give by the expectation of a consistent but biased estimator of the log-marginal likelihood, with a bias of order $\mathcal{O}(1/k)$ given by the relative variance of the importance sampling of $p_\theta(x)$. Thus, increasing the number $k$ of importance samples tightens the bound and reduces the bias, and the spread of the distribution of the importance weight fully characterizes the variational gap. 
\cite{domke2018importance}, \cite{klys2018joint}, \cite{rainforth2018nesting} and \cite{dhekane2021improving} extended and generalized the previous results, while \cite{huang2019note} and \cite{huang2019hierarchical} suggested looking at the variance of the logarithm of the importance weights in order to obtain information on the tightness of the bounds. More recently \cite{mattei2022uphill} and \cite{daudel2023alpha}  have further complemented these results.

Another interesting line of theoretical work, which we will not explore further in this paper, deals with optimization issues. Indeed, in practice, the Monte Carlo objective is optimized using stochastic gradient descent with respect to the model and the variational parameters, which can cause some problems. Perhaps counterintuitively, the performance of the IWVI as a function of the variational parameter does not necessarily improve as the number $k$ of importance samples increases, although the lower bound becomes tighter. As demonstrated in \cite{rainforth18b}, this is formalized by the fact that the signal-to-noise ratio (SNR) of the standard mini-batch estimator of the gradient with respect to the variational parameter goes to $0$ when $k\rightarrow\infty$. This is due to the fact that although the expectation of the gradient decreases, its variance decreases more slowly without ever vanishing. This issue can be mitigated using reparameterization ideas \cite{tucker2018doubly}. Additional studies of the SNR for IWVI objectives can be found in \cite{lievin2020optimal} and \cite{daudel2023alpha}.

\subsection{Maximum Simulated Likelihood Estimator}
While the theoretical study of the limit behavior of the IWVI estimate has never been performed, Maximum Simulated Likelihood Estimation (MSLE) is a closely related procedure proposed in the early 1980s \citep{lerman1981use,pakes1986patents,pakes1989simulation,mcfadden1989method} that has been thoroughly studied in both the statistics and econometrics communities, see e.g. \cite{geyer1994convergence}, \cite{gourieroux1996simulation} and \cite{train2003discrete}.

Similarly to IWVI, MSLE considers an unbiased importance sampling estimate of the likelihood $\hat{p}^{(k)}_\theta(x)$. However, instead of maximizing $\mathbb{E}[\log \hat{p}^{(k)}_\theta(x)]$ using stochastic gradient techniques, it does generate importance samples only once sand then maximizes the corresponding simulated likelihood function $\log \hat{p}^{(k)}_\theta(x)$. To the best of our knowledge, the connections between MSLE and IWVI have never been noticed previously in the literature. We highlight here the main differences between IWVI and MSLE as follows:
\begin{itemize}
    \item The objective of MSLE is a random function $\log \hat{p}^{(k)}_\theta(x)$, while IWVI maximizes a deterministic one $\mathbb{E}[\log \hat{p}^{(k)}_\theta(x)]$.
    \item IWVI optimizes jointly both the model parameter $\theta$ and the variational parameter $\phi$, while the two problems are addressed separately for MSLE.
    \item The implementation of IWVI in the machine learning community has been largely accompanied by the development of scalable black-box stochastic optimization techniques over both the model and the variational parameters, while the emphasis in econometrics has been more on the design of ad-hoc models and proposals with provable guarantees.
\end{itemize}

We review here the theoretical properties of MSLE \citep{lee1992efficiency,lee1995asymptotic,gourieroux1996simulation,sung2007monte}. The limit properties of the corresponding model estimator actually depend crucially on the order of the importance sample size $k$ used, and more precisely on the relationship between the number of draws $k$ and the observed data sample size $n$ denoted $(x_1,...,x_n)$. If $k$ is considered fixed, then the model estimator does not converge to the true parameter due to a so-called simulation bias \citep{gourieroux1996simulation,train2003discrete}. When $k$ increases with $n$, i.e. when the number of draws increases to infinity with the observed data sample size, this simulation bias disappears and the estimator is then consistent. The question of asymptotic efficiency is, however, more subtle. It depends on the rate at which $k$ grows w.r.t. $n$ and depends whether the draws are overlapping or independent, i.e. whether the same importance samples $\{z^\ell\}_{1\leq \ell\leq k}$ are used for all the observations or whether different importance samples are generated for each observation. In the case of overlapping draws, we obtain efficiency, i.e. asymptotic equivalence to the exact maximum likelihood estimator (MLE), when $k/n\rightarrow\infty$. Otherwise the simulation variance dominates the inevitable variance of the true MLE \citep{sung2007monte}. However when the draws are independent, then MSLE is efficient when $k/\sqrt{n}\rightarrow\infty$, and the estimator is not even asymptotically normal in the opposite case \citep{lee1992efficiency,lee1995asymptotic}. All these results have been established under strong regularity conditions for the simulation design. The properties of MSLE can be summarized as follows:
\begin{itemize}
    \item If $k$ is fixed and $n\rightarrow\infty$, then the estimator is inconsistent.
\end{itemize}
\begin{itemize}
    \item For overlapping draws, if both $k\rightarrow\infty$ and $n\rightarrow\infty$:
    \begin{itemize}
        \item If $k$ rises slower than $n$, then the estimator is consistent but not efficient.
        \item If $k$ rises faster than $n$, then the estimator is consistent and efficient.
    \end{itemize}
\end{itemize}
\begin{itemize}
    \item For independent draws, if both $k\rightarrow\infty$ and $n\rightarrow\infty$:
    \begin{itemize}
        \item If $k$ rises slower than $\sqrt{n}$, then the estimator is consistent but not efficient.
        \item If $k$ rises faster than $\sqrt{n}$, then the estimator is consistent and efficient.
    \end{itemize}
\end{itemize}

\subsection{Outline and contributions} 
We provide in this paper the first asymptotic results for IWVI estimates. In particular, we make the following two main contributions:
\begin{itemize}
    \item We show in Section 2 that the model parameter estimator is consistent when $n$ and $k$ both tend to infinity. Additionally, we prove that the variational parameter estimator is also consistent when $n$ and $k$ tend to infinity, and that this limit is the minimizer of the expectation of the first-order variance term in \eqref{devIWELBO}, that is of the relative variance of the importance sampling estimate of the likelood.
\end{itemize}
\begin{itemize}
   \item We use the reparameterization trick \citep{kingma2013auto} in Section 3 to prove that IWVI is even asymptotically efficient when $k$ grows fast enough. As for MSLE, we distinguish several regimes, which depend here on the smoothness of the importance ratio distribution.
\end{itemize}

All proofs and technical results are detailed in the appendices. We emphasize that the assumptions used in our study are very weak, requiring mainly the existence of low-order moments for the importance ratios and their logarithms, in addition to some minor regularity assumptions. These kinds of conditions, which are often encountered when studying the tightness of the variational bound, are much weaker that those required in the asymptotic study of MSLE.

\subsection{Notations and background} 
We close this section by introducing the precise notations we will use in the rest of the paper. We assume that $\seq{\x_i}{1\leq i\leq n}$ is a sequence of independent and identically distributed (i.i.d.) random vectors (with unknown distribution) on the probability space $(\Omega,\mcf,\PP_\star)$ that take values in a measurable space $(\Xset,\Xsigma)$ where $\Xset \subset \rset^{d_x}$. We denote by $\PE_\star$ the associated expectation. Conditionally on $\X=\seq{\x_i}{1\leq i \leq n}$, we consider a family of independent random sequences $\seq{\bz_i=(\bz_i^1,\ldots,\bz_i^\lo)}{1\leq i\leq n}$ such that $\seq{\bz_i^\ell}{1\leq \ell\leq \lo}$ are i.i.d. random vectors taking values on a measurable space $(\Zset,\Zsigma)$ where $\Zset \subset \rset^{d_z}$, with density $\vz_i \mapsto q_\phi(\vz_i|\x_i)$ with respect to (w.r.t.) the Lebesgue measure on $\rset^{d_z}$. Given a model with joint density $p_\theta(x,z)=p_\theta(x|z)p_\theta(\vz)$ w.r.t. the Lebesgue measure on $\rset^{d_x}\times\rset^{d_z}$ where $x$ denotes the observation and $z$ the latent variable, IWVI aims at finding the parameter $\theta$ which best describes the observations $\X=\seq{\x_i}{1\leq i \leq n}$ by maximizing the Monte Carlo Objective (MCO)
\begin{equation}
    \label{eq:def:xi:n:ell}
    \t \xi  \nell=(\ttheta \nell,\tphi \nell)=\arg\max\limits_{\xi=(\theta,\phi)} \lk \nell (\xi) \quad \mbox{where}  \quad \lk \nell (\xi)\eqdef\sum_{i=1}^n \PE^\phi_\X\lrb{\log\lr{\frac 1 k \sum_{\ell=1}^k \frac{p_\theta(\x_i,\bz_i^\ell)}{q_\phi(\bz_i^\ell|\x_i)}} }, 
\end{equation}
where the notation $\PE^\phi_\X$ stands for the conditional expectation of $\seq{\bz_i=(\bz_i^1,\ldots,\bz_i^\lo)}{1\leq i\leq n}$ w.r.t. $\X=\seq{\x_i}{1\leq i \leq n}$. The complete IWVI parameter $\t \xi  \nell=(\ttheta \nell,\tphi \nell)$ is maximized over $\Xi=\Theta \times \Phi$ where $\Theta$ and $\Phi$ are two compact sets in $\rset^{d_\theta}$ and $\rset^{d_\phi}$. 
Note that we do not necessarily assume that the true data generating process $\PP_\star$ is contained in the given model of density $p_\theta(x)=\int_\Zset p_\theta(x|z)p_\theta(\vz) dz$ nor that $\PP_\star$ is dominated by the Lebesgue measure.

In what follows, for any $\ell \in\nset$, $|\cdot|$ denotes the Euclidean norm on $\rset^\ell$ (i.e. we suppress the dependence on the dimension $\ell$ in the notation when no ambiguity occurs) and for any $\rset^\ell$-valued random vector and any $s>0$, 
\begin{equation} 
    \normsx{U}{\phi}{s}=(\PE^\phi_\X[|U|^s])^{1/s} .
\end{equation}  

\section{Consistency}
This section aims to establish the strong consistency of both the IWVI model parameter $\ttheta \nell$ and variational parameter $\tphi \nell$. Note that both $n$ and $\lo$ go to infinity simulateneously in this section. Proofs are deferred to the appendices.

\subsection{Consistency of the model parameter}

Since the IWVI objective tends to the exact log-marginal likelihood function as $\lo\rightarrow\infty$ \citep{maddison2017filtering,nowozin2018debiasing}, it may seem natural to think that the IWVI model parameter $\ttheta \nell$ is consistent under relevant conditions when both $n \wedge \lo \to \infty$, as it is the case for MSLE. To properly obtain the convergence of $\ttheta \nell$ to the maximizer $\thv$ of the expected log-likelihood when both the observed data and the Monte Carlo sample sizes go to infinity simultaneously, we introduce the following assumptions:
\begin{hyp}{A}
    \item \label{hyp:theta:star}
    There exists $\thv \in \Theta$ such that 
    $$ 
     \{\thv\}=\arg\max\limits_{\theta \in \Theta} \PE_\star \lrb{\log p_\theta(\x)}.
    $$
     
\end{hyp}

\begin{hyp}{A}
    \item \label{hyp:q}
    For all $(\theta,\phi) \in \Theta \times \Phi$, 
    $$ 
     \PE_\star\lrb{\int q_\phi(\vz|\x) \log^+ \lr{\frac{q_\phi(\vz|\x)}{p_\theta(\vz|\x)}} \rmd z} <\infty.
    $$
\end{hyp}

\begin{hyp}{A}
    \item \label{hyp:unif}
    \begin{enumerate}[(i)]
        \item $\PP_\star-\as$, the function $\theta\mapsto p_\theta(\x)$ is upper-semi
          continuous. 
        \item $\PE_\star\lrb{\sup\limits_{\theta\in \Theta} \log^+ p_\theta(\x)}<\infty$. 
    \end{enumerate}
\end{hyp}

Assumption \ref{hyp:theta:star} is frequently employed to establish the consistency and asymptotic normality of maximum likelihood estimators in the setting of completely observable variables \citep{VanDerVaart1999}. This assumption is natural and postulates the existence of a unique parameter $\thv$ that maximizes the expected maximum likelihood, corresponding to the best approximation in the model of the data generating process $\PP_\star$ as measured by the Kullback-Leibler distance. When the model is correctly specified, this parameter is the true parameter itself. Assumption \ref{hyp:q} is a mild assumption that ensures the finiteness of a slight variant of the Kullback-Leibler divergence between the variational and posterior distributions, which is necessary for the existence of the quantities being considered. This assumption is typically satisfied when the variational and posterior distributions are isotropic Gaussians respectively centered at $\phi$ and $\theta$. Another example satisfying this assumption is the Gaussian linear setting introduced by \cite{rainforth18b}, where the prior distribution over $\bz$, whose density is denoted $p_\theta(\cdot)$, is a Gaussian $\mathcal{N}(\theta,I_{d_\bz})$ with $\theta\in\mathbb{R}^{d_\bz}$, the conditional density $p_\theta(\cdot|\bz)$ follows a Gaussian $\mathcal{N}(\bz,I_{d_\x})$, and $q_\phi(\cdot|\x)$ is the density of a Gaussian $\mathcal{N}(A\x+b,2/3 I_{d_z})$, where $A=\textnormal{diag}(a)$ is a diagonal matrix and $\phi=(a,b)\in\mathbb{R}^{d_\bz}\times\mathbb{R}^{d_\bz}$. Finally, Assumption \ref{hyp:unif} is a standard assumption required to prove the consistency of the regular MLE for general models, see \cite{Ferguson1996}.

We are now ready to present Theorem \ref{thm:consistency}, which is proved in Appendix \ref{proof:thm:consistency}.

\begin{theorem}
    \label{thm:consistency}
    Assume \ref{hyp:theta:star}-\ref{hyp:q}-\ref{hyp:unif}. Then, $\PP_\star-\as$,
    $$ 
    \lim_{n\wedge \lo \to \infty} \t \theta \nell = \thv.
    $$
\end{theorem}

We would like to emphasize the strength of our result. In addition to standard conditions for consistency of the maximum likelihood such as Wald's formulation \citep{Ferguson1996}, our approach only requires Assumption \ref{hyp:q}, which is mild and ensures the existence of the objective and related expectations. This stands in contrast to existing consistency results for MSLE estimators. For instance, in the overlapping case, \cite{sung2007monte} utilize intricate technical arguments concerning the weak convergence of stochastic processes, which entail demanding uniform strong laws of large numbers. Similarly, for independent draws, the proof of \cite{lee1992efficiency} involves high-order derivatives of the log-simulated likelihood function and simulation errors, requiring strong differentiability and moments conditions on the log-simulated likelihood and on the simulator. Furthermore, the proof of \cite{lee1992efficiency} also depends on the assumption that the likelihood function is strictly bounded away from zero, which is a very strong requirement.

\subsection{Consistency of the variational parameter}

Our investigation also includes studying the asymptotic limit of the IWVI variational parameter $\tphi \nell$.

We expect the limit of $\tphi \nell$ to correspond to the minimizer of the dominant first-order term involving the variational parameter in the development of \eqref{devIWELBO} in powers of $1/\lo$, which takes the form of a variance. We demonstrate below that under relevant conditions, $\tphi \nell$ is indeed strongly consistent and converges to such a variance minimizer that we define properly. Let us first introduce the ratio of the importance weight $p_\theta(\bz,\x)/q_\phi(\bz|\x)$ and of the marginal density $p_\theta(\x)$, as well as its variance
\begin{equation}
  \label{eq:defVar}
  r_\xi^\x(\bz):=\frac{p_\theta(\bz,\x)}{q_\phi(\bz|\x)p_\theta(\x)}, \quad \Vratio_\xi(\x)\eqdef\PE^\phi_\X \lrb{
    \lr{r_\xi^\x(\bz)-1}^2}  \quad \mbox{and} \quad \Vratio(\xi) \eqdef \PE_\star\lrb{\Vratio_\xi(\x)} .
\end{equation}

We then make the following assumptions:

\begin{hyp}{A}
    \item \label{hyp:moment} 
 $           \PE_\star \lrb{\sup_{\xi \in\Xi}
          \lr{\normsx{r_\xi^\x(\bz)}{\phi}{4}^3
          \lr{\normsx{\log(r_\xi^\x(\bz))}{\phi}{4}+ 1} }} <\infty .
 $
\end{hyp}

\begin{hyp}{A}
\item \label{hyp:variational}
  \begin{enumerate}[(i)]
  \item \label{hyp:consist:general:minimizer} There exists $\phiv \in\Phi$ such that
    $$
    \{\phiv\}=\argmin_{\phi\in\Phi}   \Vratio (\thv,\phi) .
    $$
  \item \label{hyp:consist:general:continuity} $\PP_\star-\as$, the function $\xi \mapsto \Vratio_\xi(\x)$ is continuous, where $\Vratio_\xi$ is defined in \eqref{eq:defVar}.  
  \end{enumerate}
\end{hyp}

Assumption \ref{hyp:moment} requires the existence of finite moments for the importance ratios, which is a standard assumption in the IWVI literature. Of course, a second-order moment is necessary to ensure the existence of the variance of the importance weight (the first-order term in Expansion \eqref{devIWELBO}), and higher-order moments have been required to derive such developments (a finite $6$-th order central moment is for instance required in \cite{maddison2017filtering}, while \cite{nowozin2018debiasing} needs finite moments of all orders). Assumption \ref{hyp:variational} respectively ensures the uniqueness of the first-order variance minimizer with respect to $\phi$ when evaluated at the true $\theta^*$ \ref{hyp:consist:general:minimizer} and the continuity of the variance \ref{hyp:consist:general:continuity}. These conditions may be non-standard for importance-weighted variational inference, as they are not satisfied for complex non-smooth deep generative models, but are important for achieving consistency, as for the regular model parameter.

With these assumptions in place, we can now establish the strong consistency of $\tphi \nell$ to the variance minimizer.

\begin{theorem}
\label{thm:consistency-variational}
  Assume
  \ref{hyp:theta:star}-\ref{hyp:q}-\ref{hyp:unif}-\ref{hyp:moment}-\ref{hyp:variational}. Then,
  $\PP_\star-\as$,
  $$
    \lim_{n\wedge \lo \to \infty} \t \phi \nell = \phiv.
  $$
\end{theorem}

To the best of our knowledge, this is the first result that establishes the consistency of the variational parameter to a variance minimizer in the IWVI framework. The proof can be found in \Cref{proof:thm:variational}.

\section{Asymptotic normality}

In this section, we show that IWVI is also asymptotically normal and even efficient under additional assumptions when $\lo$ grows fast enough to infinity (relative to $n$).

\subsection{Reparameterization}

We work in this section under an additional mild assumption that is commonly used in practice to compute reliable stochastic estimates of the lower bound and to make variational inference suitable for backpropagation in deep generative models. Let us introduce it formally.

\begin{hyp}{A}
\item \label{hyp:reparam} There is a measurable space $(\Eset, \Esigma)$ with $\Eset \subset \rset^{d_\epsilon}$, a measure on $(\Eset, \Esigma)$ of density $\nu$ w.r.t. the Lebesgue measure on $\rset^{d_\epsilon}$, and a measurable function $g_\phi: \Xset \times \Eset \to \rset^+$  such that for any $\x \in \Xset$, if $\epsilon  \sim \nu$ and $\bz=g_\phi(\x,\epsilon)$, then $\bz\sim q_\phi(\cdot|\x)$. 
\end{hyp}

Assumption \ref{hyp:reparam} means in particular that for any non-negative measurable function $h$ on $(\Zset,\Zsigma)$,
$$
\int_\Zset h(\bz) q_\phi(\bz|\x) \rmd z=\int_\Eset h(g_\phi(\x,\epsilon)) \nu(\epsilon) \rmd \epsilon\eqsp. 
$$
This identity implies the change of measure formula: $q_\phi(g_\phi(\x,\epsilon)|\x) |\mathrm{Det} \lr{\partial_\epsilon   g_\phi(\x,\epsilon)}|=\nu(\epsilon)$. This alternative way to obtain a sample $\bz\sim q_\phi(\cdot|\x)$ by using a random variable $\epsilon$ with a fixed distribution $\nu$ is the basis of the reparametrization trick, widely used in the variational inference literature \citep{kingma2013auto,IWAE2015}. 

In what follows, we express the random variables $\seq{\bz_i=(\bz_i^1,\bz_i^2,\ldots)}{i\in \nset}$ in terms of the random variables  $\seq{\beps_i=(\epsilon_i^1,\epsilon_i^2,\ldots)}{i\in \nset}$ using $\bz_i^\lo=g_\phi(\x_i,\epsilon_i^\lo)$ where $\seq{\epsilon_i^\lo}{i,\lo \in\nset}$ are i.i.d.\ with distribution $\nu$. This allows us to consider an expectation $\PE^\phi_\X$ conditional on $\X$ that does not actually depend on $\phi$ (since the distribution $\nu$ does not depend on the parameter $\phi$, unlike $q_\phi$). As a consequence, we replace $\PE^\phi_\X$ by $\PE_\X$ whenever we consider the expectation conditional on $\X$ of functions of $\seq{\epsilon_i^\lo}{i,\lo \in\nset}$ and $\X$ only.

Using this notation, we can now get another expression for $\lk \nell$: 
\begin{align}
    \lk \nell(\xi)&= \sum_{i=1}^n \log p_\theta(\x_i)+ \sum_{i=1}^n \PE_\X \lrb{\log \lr{\frac 1 \lo \sum_{\ell=1}^\lo \frac{p_\theta(g_\phi(\x_i,\epsilon_\ell)|\x_i)}{q_\phi(g_\phi(\x_i,\epsilon_\ell)|\x_i)}}}\notag\\
& \stackrel{(a)}{=} \sum_{i=1}^n \log p_\theta(\x_i)+ \sum_{i=1}^n \PE_\X \lrb{\log \lr{\frac 1 \lo \sum_{\ell=1}^\lo \frac{f_\xi(\epsilon_\ell|\x_i)}{\nu(\epsilon_\ell)}}},  \label{eq:def:L} 
\end{align}
where 
$$ f_\xi(\epsilon_\ell|\x_i)=p_\theta(g_\phi(\x_i,\epsilon_\ell)|\x_i) |\mathrm{Det} \lr{\partial_\epsilon g_\phi(\x_i,\epsilon_\ell)}| . 
$$
In $\stackrel{(a)}{=}$, we have used the change of variable formula:
$q_\phi(g_\phi(\x,\epsilon)|\x) |\mathrm{Det} \lr{\partial_\epsilon
g_\phi(\x,\epsilon)}|=\nu(\epsilon)$.  

\subsection{Efficiency of $\t \theta \nell$}

We now let $\lo$ depend explicitly on $n$ and therefore write $\lo_n$ instead
of $\lo$. This allows us to control the rate at which $k$ increases relatively to $n$. The following standard assumptions will be needed:

\begin{hyp}{A}
    \item \label{hyp:consist}
      There exists $\xiv= (\thv,\phi^\star)\in\mathring{\Xi}$ such that $\t \xi \nell \convprob{\PP_\star} \xiv$, as $n \wedge k \to \infty$.
    \end{hyp}
    \begin{hyp}{A}
    \item \label{hyp:as:zero}
      \begin{enumerate}[(i)]
      \item $\theta \mapsto p_\theta(\x)$ is twice differentiable on
        $\Theta$.
       \item \label{item:moment:unif} $\PE_\star \lrb{\sup\limits_{\theta \in \Theta} \norm{\nabla^2_\theta \log p_\theta(\x)}}  < \infty $. 
      \end{enumerate}
      \end{hyp}
       
    Define for positive real numbers $\alpha, \delta$ such that $\alpha^{-1}+\delta^{-1}=1$,  
   \begin{align*}
       M_{0,\xi}(\x)&= \lrb{1+\lr{\normsxsimp{\frac{f_\xi(\epsilon|\x)}{\nu(\epsilon)}}{2\delta}}^2} \normsxsimp{\nabla_\xi \log f_\xi(\epsilon|\x)}{\alpha} ,&  N_{0,\xi}(\x) &=\lr{1+\normsxsimp{\frac{f_\xi(\epsilon|\x)}{\nu(\epsilon)}}{2\delta}} \normsxsimp{ \frac{\nabla_\xi f_\xi(\epsilon|\x)}{\nu(\epsilon)} }{2},\\
   M_{1,\xi}(\x)&=\lrb{1+\lr{\normsxsimp{\frac{f_\xi(\epsilon|\x)}{\nu(\epsilon)}}{2\delta}}^2} \normsxsimp{\frac{\nabla_\xi^2 f_\xi(\epsilon|\x)}{f_\xi(\epsilon|\x)}}{\alpha} ,& N_{1,\xi}(\x) &=\lr{1+\normsxsimp{\frac{f_\xi(\epsilon|\x)}{\nu(\epsilon)}}{2\delta}} \normsxsimp{ \frac{\nabla_\xi^2 f_\xi(\epsilon|\x)}{\nu(\epsilon)} }{2},\\
   M_{2,\xi}(\x)&=\normsxsimp{\frac{f_\xi(\epsilon|\x)}{\nu(\epsilon)} -1}{\frac{\alpha}{\alpha-2} \vee 2}  \times \normsxsimp{\nabla_\xi \log f_\xi(\epsilon|\x)}{\alpha}
   ,&N_{2,\xi}(\x)&=\normsxsimp{\frac{\nabla_\xi f_\xi(\epsilon|\x)}{\nu(\epsilon)}}{2}.
   \end{align*}
   Obviouly, $M_{i,\xi}(\x)$ and $N_{i,\xi}(\x)$ for $i \in \{0,1,2\}$ depend on $\alpha$ but for simplicity, we make implicit this dependence and do not stress it in the notation.  
   \begin{hyp}{A}
   \item \label{hyp:weak:one}  For any $(\epsilon,\x) \in \Eset \times \Xset$, the function $\xi \mapsto f_\xi(\epsilon|\x)$ is twice differentiable and there exists $\alpha>2$ such that, setting
     $\delta=\alpha/(\alpha-1) \in (1,2)$, 
     \begin{enumerate}[(i)]
     \vspace{0.2cm}
       \item $\lim\limits_{n \to \infty}\lo_n/n^{\delta/2}=\infty$,
     \vspace{0.2cm}
       \item    $\PE_\star[M_{0,\xi^\star}(\x)]+\PE_\star[N_{0,\xi^\star}
    (\x)] <\infty$,
     \vspace{0.2cm}
      \item 
      $ \PE_\star\lrb{\sup\limits_{\xi \in \Xi} M_{1,\xi}(\x)}+\PE_\star\lrb{\sup\limits_{\xi \in \Xi} N_{1,\xi}
   (\x)}<\infty$,
     \vspace{0.2cm}
     \item      
     $ \PE_\star\lrb{\sup\limits_{\xi \in \Xi} M^2_{2,\xi}(\x)}+\PE_\star\lrb{\sup\limits_{\xi \in \Xi} N^2_{2,\xi}(\x)}<\infty $.
     \end{enumerate}
     
    \end{hyp}

\noindent We now state our theorem, which proof can be found in \Cref{proof:thm:efficiency}.

\begin{theorem}
    \label{thm:AS}
    Assume \ref{hyp:reparam}-\ref{hyp:consist}-\ref{hyp:as:zero}-\ref{hyp:weak:one}. Then, provided that the matrix 
    $$
    J_{1}\eqdef\PE_\star \lrb{\nabla_\theta \log p_\thv(\x)\lr{\nabla_\theta \log p_\thv(\x)}^T}
    $$  
is non-singular, we have 
    $$ 
        n^{1/2}(\t \theta^{k_n}_n - \thv) \weakconv_{\PP_\star} \mathcal{N}(0,J_{2}^{-1}J_{1}J_{2}^{-1}),
    $$
    where $J_{2}\eqdef\PE_\star \lrb{\nabla_\theta^2 \log p_\thv(\x)}$.
\end{theorem}

We demonstrate that asymptotic normality and efficiency of our estimator are achieved provided $k$ is sufficiently large. Notably, while $k$ must diverge at a rate faster than $\sqrt{n}$ for the independent MSLE and at least $n$ for the overlapping MSLE, the required growth of $k$ for the IWVI estimator lies between these regimes. Specifically, the MSLE with independent draws and the MSLE with overlapping draws are two variants of the same estimator, differing only in the method of sampling latent variables. In contrast, the IWVI estimator of interest is nearly identical to the MSLE but treats latent variables in expectation rather than through sampling. Consequently, it is natural for the IWVI estimator to require a threshold for $k$ that lies between $\sqrt{n}$ and $n$, reflecting a phase transition dependent on the moments of the importance weights. 

While Assumptions \ref{hyp:consist} and \ref{hyp:as:zero} align with standard requirements for the asymptotic normality of the regular MLE, the exact threshold for $k$ is dictated by Assumption \ref{hyp:weak:one}, which is common in the IWVI literature and mainly ensures the existence of finite moments for the importance weights, their logarithms and their first- and second-order derivatives. Remarkably, Assumption \ref{hyp:weak:one} even relaxes the more demanding conditions employed in \cite{lee1992efficiency,lee1995asymptotic,sung2007monte}, which rely on uniformity arguments. This refined analysis reveals the phase transition between the $\sqrt{n}$ and $n$ thresholds, allowing us to identify these intermediate thresholds for $k$.

\section{Simulation study}

In this section, we provide a numerical study of $\tilde{\theta}_n^k$ in a simple but intractable unobserved heterogeneity model, which is widely used in econometrics for its ability to account for individual differences not captured by observed variables and that can distort the estimated effect of covariates on outcomes like survival times or hazard rates. While the assumptions underpinning the asymptotic theories of IWVI and MSLE may not strictly hold in this specific setup, the numerical experiments serve to complement our theoretical analysis. In particular, they allow us to explore the practical performance of $\tilde{\theta}_n^k$ and to compare it to the classical MSLE estimators with both independent and overlapping draws, highlighting similarities and differences in behavior across these approaches.

\vspace{0.2cm}

The model considers independent Gaussian random variables $\x_i |\bz_i \sim \mathcal{N}(\theta+\bz_i, 1)$, where $\bz_i$ represents unobserved heterogeneity that is assumed to have a known skewed Gumbel distribution of density $q(z) = e^{-z} \exp(-e^{-z})$. Maximum likelihood estimation is not possible as the marginal density has no closed-form solution, and we instead use IWVI using the direct simulator $q(z)$, so that $\tilde{\theta}_n^k$ maximizes

$$
\tilde{\theta}_n^k = \arg\max_\theta \sum_{i=1}^n \PE_\X\lrb{\log\lr{\frac 1 k \sum_{\ell=1}^k \frac{1}{\sqrt{2\pi}} \exp \left( -\frac{(\x_i - \theta - \bz_i^\ell)^2}{2} \right)} }, 
$$
where $\PE_\X$ simply stands for the expectation of the independent latent random variables $\seq{\bz_i=(\bz_i^1,\ldots,\bz_i^\lo)}{1\leq i\leq n}$ following the Gumbel distribution.

\vspace{0.2cm}

The main objectives of this numerical section are:
\begin{enumerate}
\item to illustrate how $\tilde{\theta}_n^k$ gets closer and closer to the MLE as the number of draws $k$ increases.
\item to evaluate the accuracy of $\tilde{\theta}_n^k$ and identify the source of error (bias/variance) in different regimes as a function of $k$ and $n$.
\item to provide a comparison of the IWVI estimator $\tilde{\theta}_n^k$ with the MSLE for both independent and overlapping draws.
\end{enumerate}

\subsection{Cameron-Trivedi's dataset} The first dataset we use was provided by \cite{cameron2005microeconometrics} in their book \textit{Microeconometrics: Methods and Applications} and is available on the book's website. It was generated based on the model described above with $\theta^*=1$ and $n=100$. The log-likelihood for this model can be approximated precisely through numerical integration.

\vspace{0.2cm}

We first investigate the MSLE estimator. Figure \ref{fig1:mainfig} presents the distribution of the estimated MSLE across multiple replications, using independent draws in Figure \ref{fig1:subfig1} and overlapping draws in Figure \ref{fig1:subfig2}. For a given value of $k$, the distribution of the MSLE is displayed using a boxplot resulting from 500 replications of the experiments, i.e.\ 500 different sets of latent variables of size $k$ are generated, and a MSLE estimate is computed for each set. We observe that the distribution of the MSLE concentrates around the MLE (indicated by a dashed line) as $k$ increases, and that as predicted by the existing literature, independent draws induce a large bias for smaller values of $k$, while overlapping draws induce a large variance for smaller values of $k$.

\vspace{0.2cm}

The conclusions are slightly different for the IWVI estimator $\tilde{\theta}_n^k$, see Figure \ref{fig1:subfig3}. Indeed, it is a deterministic estimator conditional on the observed dataset $(\x_i)_{i=1}^n$, contrary to the MSLE where the randomness comes from both the observed dataset and from the set of latent variables. The boxplots correspond here to the fact that the expectation in the definition of $\tilde{\theta}_n^k$ is numerically approximated by a sample mean over $10^6$ sets of $k$ draws of the latent variables. The tightness of the boxes suggests that this is a sufficient number of replications to approximate the expectation. Notice that the error (the squared bias) resulting from taking any finite value of $k$ is always smaller than the corresponding error (squared bias + variance) of the MSLE, which suggests that we can benefit from removing the randomness in our objective.

\begin{figure}[h]
\centering
\begin{subfigure}{0.3\textwidth}
  \includegraphics[width=\textwidth]{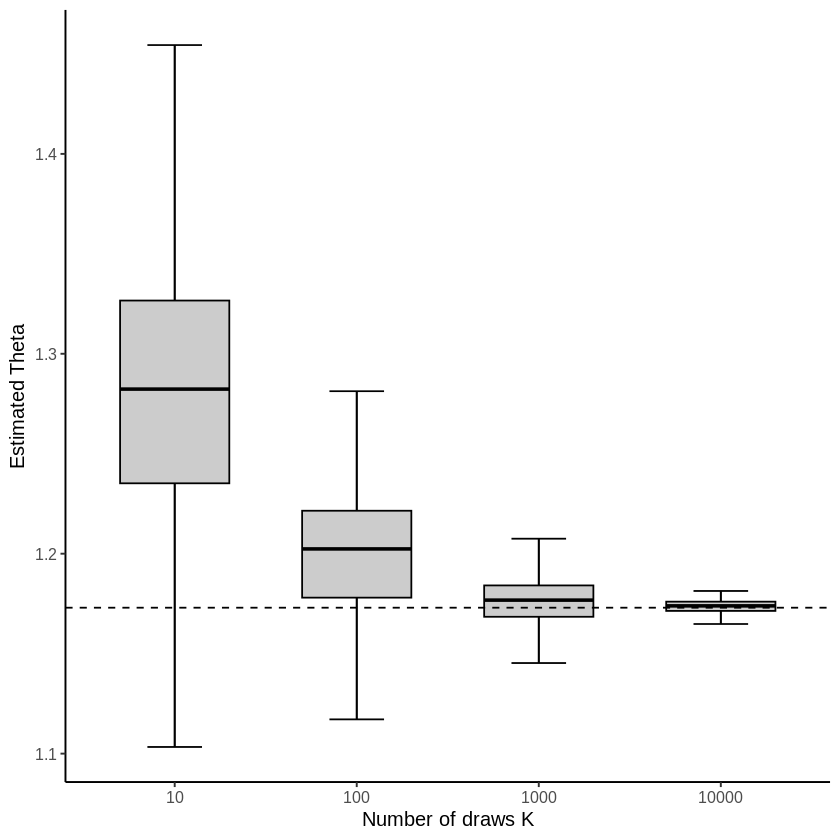}
  \caption{MSLE - independent draws}
  \label{fig1:subfig1}
\end{subfigure}
\hfill
\begin{subfigure}{0.3\textwidth}
  \includegraphics[width=\textwidth]{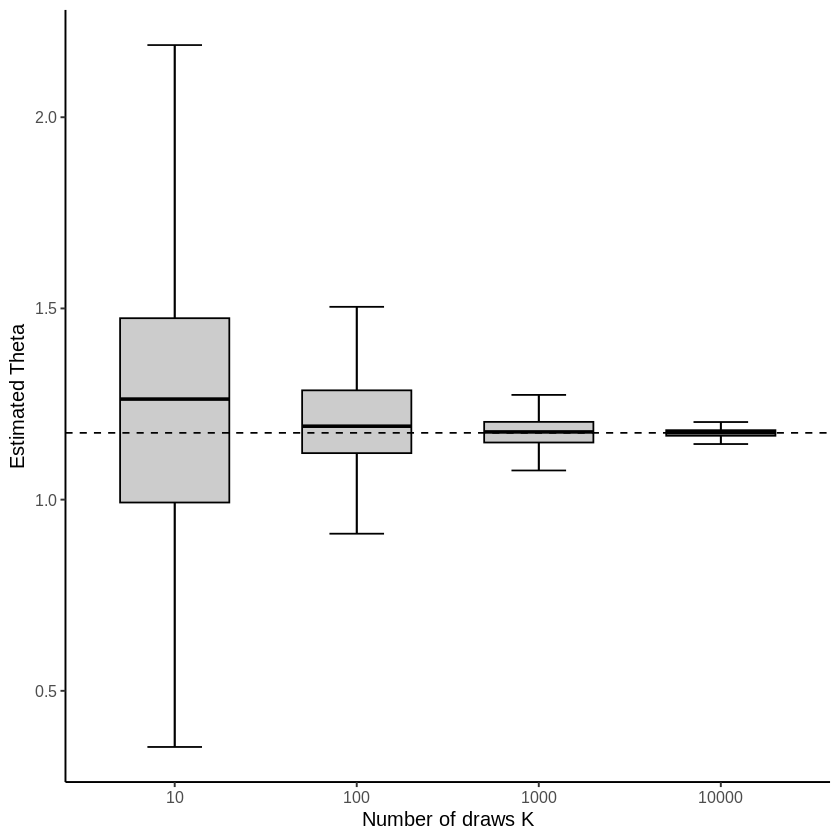}
  \caption{MSLE - overlapping draws}
  \label{fig1:subfig2}
\end{subfigure}
\hfill
\begin{subfigure}{0.3\textwidth}
  \includegraphics[width=\textwidth]{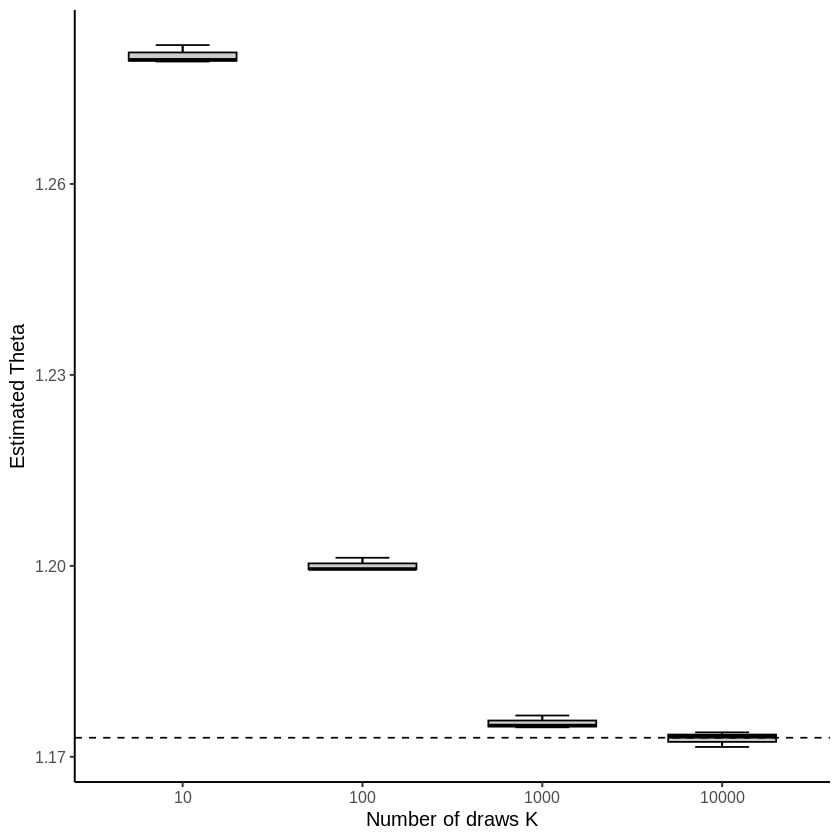}
  \caption{IWVI}
  \label{fig1:subfig3}
\end{subfigure}
\caption{Boxplots of MSLE and IWVI estimates over 500 replications of the draw of the latent set of size $k$ based on Cameron \& Travedi's dataset with $n = 100$ observations. The dashed line represents the maximum likelihood value. One can see that increasing the number of draws $k$ improves the estimation accuracy. Note that the scales are different in the three figures.}
\label{fig1:mainfig}
\end{figure}

\subsection{Simulated datasets} We now try to incorporate randomness in both the observed $x$ and the latent $z$ variables in our experiments. To do so, we generate $n=100$ datapoints from the same model with $\theta^*=1$, and we then perform inference on $\theta$ using both the MSLE (independent and overlapping) from random draws of latent variables and the IWVI estimate $\tilde{\theta}_n^k$. Across 500 replications of this procedure, we compare the Mean Square Error (MSE) of the estimates for the two MSLE estimators and for the IWVI, as detailed in Table \ref{table:MSLEvsIWVI}, with a particular focus on the bias and variance decomposition of the MSE: the variance part of the error is indeed indicated in Table \ref{table:MSLEvsIWVI} for each estimator. The findings are interesting as they provide valuable insights in the light of previous results in the simulated likelihood literature:

\begin{itemize}
    \item The MSLE with independent draws becomes more accurate as $k$ increases, with a non-ngeligible bias (measured by the difference between the MSE and the variance in brackets) for moderate values of $k$, and the MSE decreases rapidly and then stagnates. These observations are supported by the theory according to which the estimator is asymptotically biased when $k\lesssim\sqrt{n}$ and efficient as soon as $k\gtrsim\sqrt{n}$ (strictly).
    \item The MSLE with overlapping draws does not suffer from a large bias for small values of $k$, and improvements in the MSE are still important for large values of $k$. This is respectively explained by the asymptotic unbiasedness of the estimator, and by its efficiency in the $k\gtrsim n$ regime only. Note that the overall results are worse than for independent draws, as for Cameron-Trivedi's dataset.
    \item The IWVI estimator is quite similar to the MSLE with independent draws, with a nonnegligible bias for small values of $k$ and a quick decay of the MSE: $\tilde{\theta}_n^k$ thus does not require a very large number of draws $k$ to be asymptotically equivalent to the MLE. For example, for $k=500$, when constructing for each of the $500$ randomly generated datasets a $95\%$-confidence interval centered at $\tilde{\theta}_n^k$ and based on a plug-in estimate of the MLE asymptotic variance, $\theta^*$ turns out to belong to $94.6\%$ of them, suggesting that the asymptotics seem to work well at these sample sizes and that the efficiency regime is already achieved, as for the MSLE with independent draws ($94.8\%$ empirical coverage with the plug-in estimate of the asymptotic variance). At the opposite, the MSLE with overlapping draws only have $92.8\%$ empirical coverage when using a plug-in estimate of the MLE asymptotic variance (which is not the only part of the asymptotic variance of the overlapping MSLE), which suggests that we are in the intermediate regime $\sqrt{n}\lesssim k\lesssim n$ where the overlapping MSLE is not efficient yet.
\end{itemize}

\begin{table}[ht]
\centering
\begin{tabular}{c c c c c c c c c}
\hline\hline
$k$ & 10 & 20 & 50 & 100 & 200 & 500 & 1000 & 2000 \\ 
\hline
MSLE (ind) & 0.0395 & 0.0322 & 0.0277 & 0.0261 & 0.0256  & 0.0237 & 0.0234 & 0.0232 \\ 
            & (0.0291) & (0.0278) & (0.0263) & (0.0252) & (0.0244) & (0.0235) & (0.0234) & (0.0232) \\ 
\hline
MSLE (over) & 0.1687 & 0.0904 & 0.0508 & 0.0376 & 0.0306 & 0.0256 & 0.0245 & 0.0239 \\ 
            & (0.1643) & (0.0871) & (0.0499) & (0.0374) & (0.0305) & (0.0255) & (0.0244) & (0.0239) \\ 
\hline
IWVI & 0.0396 & 0.0326 & 0.0267 & 0.0250 & 0.0247 & 0.0236 & 0.0233 & 0.0232 \\ 
            & (0.0292) & (0.0261) & (0.0254) & (0.0243) & (0.0239) & (0.0235) & (0.0233) & (0.0232) \\ 
\hline
\end{tabular}
\caption{MSE of the IWVI estimator compared with the two versions (independent and overlapping draws) of the MSLE for several values of $k$ with $n=100$, over 500 repetitions of the experiment. (The term shown in parentheses below the MSE corresponds to the variance part of the MSE, excluding the bias, over the 500 repetitions.)}
\label{table:MSLEvsIWVI}
\end{table}

\subsection{Unbiasedness of the ELBO maximizer} 
We highlight here a phenomenon that can sometimes occur for $k=1$, and was observed in our experiments on Cameron-Trivedi's dataset. Additional experiments on this dataset show that the three estimators are actually already very good when $k=1$. Indeed,  all of them appear to be centered at the MLE for $k=1$ on Figure \ref{fig2:mainfig}, and $\tilde{\theta}_n^1$ even equal to the MLE (due to its nonrandomness) as observed in Figure \ref{fig2:subfig3}. 
This is the case because the expected ELBO maximizer $\tilde{\theta}_\infty^1$ turns out to be $\theta^*$ in our model (see the derivations in the appendix), and the ELBO maximizer $\tilde{\theta}_n^1$ for a finite dataset $(\x_i)_{1\leq i \leq n}$ can be computed in closed-form as $\frac{1}{n}\sum_{1\leq i \leq n}\x_i-\gamma$, where $\gamma$ is Euler's constant. We make similar observations for expectations of the MSLE. Note however that unbiasedness does not extend to values of $k$ different from 1, see Figure \ref{fig2:mainfig}.

\begin{figure}[h]
\centering
\begin{subfigure}{0.3\textwidth}
  \includegraphics[width=\textwidth]{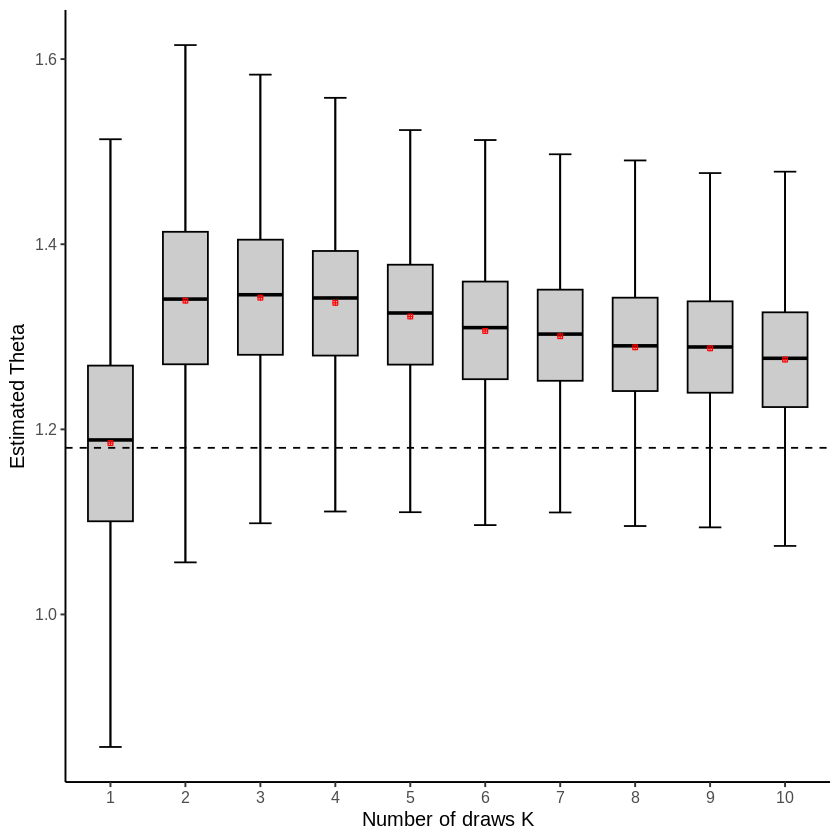}
  \caption{MSLE - independent draws}
  \label{fig2:subfig1}
\end{subfigure}
\hfill
\begin{subfigure}{0.3\textwidth}
  \includegraphics[width=\textwidth]{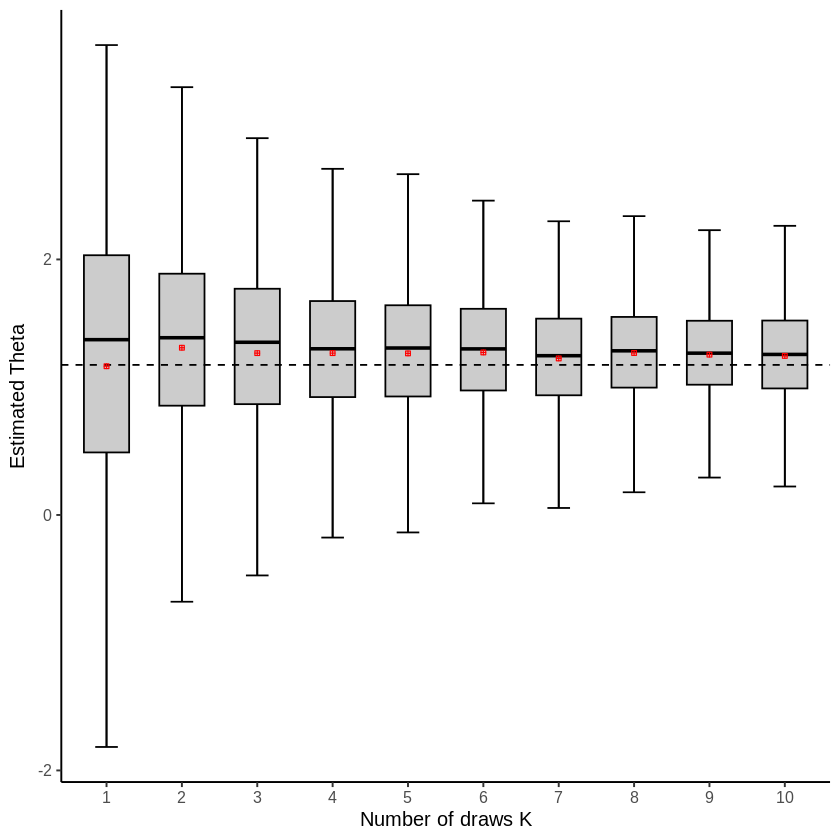}
  \caption{MSLE - overlapping draws}
  \label{fig2:subfig2}
\end{subfigure}
\hfill
\begin{subfigure}{0.3\textwidth}
  \includegraphics[width=\textwidth]{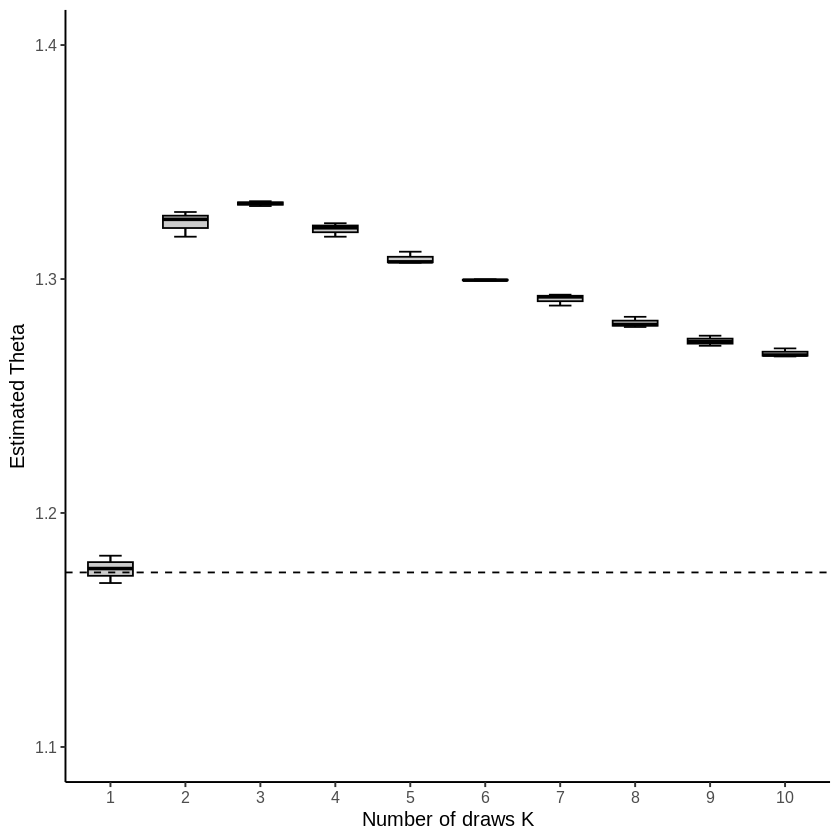}
  \caption{IWVI}
  \label{fig2:subfig3}
\end{subfigure}
\caption{Boxplots of MSLE and IWVI estimates over 500 replications of the latent set of size $k$ based on Cameron \& Travedi's dataset with $n = 100$ observations. The dashed line represents the maximum likelihood value, and the red points the mean of each boxplot. One can see that the three estimators are centered at the MLE for $k=1$.}
\label{fig2:mainfig}
\end{figure}

\begin{appendix}

\section{Proof of Theorem \ref{thm:consistency}}
\label{proof:thm:consistency}

Let us first define some preliminary notations. Under \ref{hyp:q}, we can define for every $\xi=(\theta,\phi)  \in \Xi$,  
\begin{equation}
    \label{eq:def:Dnell}
    D \nell(\xi)\eqdef - n^{-1} \sum_{i=1}^n \PE^\phi_\X\lrb{\log\lr{\lo^{-1} \sum_{\ell=1}^\lo r_\xi^{\x_i}(\bz_i^\ell)} } \quad  \mbox{where} \quad  r_\xi^\x(\bz)=\frac{p_\theta(\bz|\x)}{q_\phi(\bz|\x)} ,
\end{equation}
which is equal the variational gap divided by $n$, so that the following decomposition holds
\begin{equation} \label{eq:decomp:elbo}
    n^{-1}\lk  \nell (\xi)=n^{-1} \sum_{i=1}^n \log p_\theta(\x_i)-D \nell (\xi) .
\end{equation}
Now, Jensen's inequality applied to the log function yields 
\begin{equation}
    \label{eq:rnk:negatif}
    -D \nell(\xi) \leq n^{-1}\sum_{i=1}^n \log\ \PE^\phi_\X\lrb{\frac 1 \lo \sum_{\ell=1}^\lo r_\xi^{\x_{i}}(\bz_i^\ell) } = n^{-1}\sum_{i=1}^n \log \lr{\frac 1 \lo \sum_{\ell=1}^\lo \PE^\phi_\X\lrb{r_\xi^{\x_{i}}(\bz_i^\ell) }}=0 ,
\end{equation}
which in turn implies the nonnegativity of the variational gap via \eqref{eq:decomp:elbo}:
\begin{equation}
    \label{eq:minor:elbo} 
    n^{-1}\lk  \nell (\xi) \leq n^{-1} \sum_{i=1}^n \log p_\theta(\x_i) .
\end{equation}

Before proving Theorem \ref{thm:consistency}, we need the following lemma ensuring the convergence of the MCO evaluated at $(\thv,\phi)$ to the maximum of the expected log-likelihood, uniformly over $\phi\in\Phi$.

\begin{lemma}
    \label{lem:consistency:one}
    Assume \ref{hyp:theta:star}-\ref{hyp:q}. Then, for every $\phi\in\Phi$, we have $\PP_\star-\as$, 
    $$ 
     \lim_{n\wedge \lo \to \infty} n^{-1} \lk  \nell (\xivp{\phi})=\PE_\star\lrb{\log p_\thv(\x)} \quad \mbox{where} \quad \xivp{\phi}\eqdef (\thv,\phi) .
    $$
\end{lemma}
\begin{proof}
    By \eqref{eq:decomp:elbo}, for every $\phi\in\Phi$,
    \begin{equation*}
       n^{-1} \lk \nell (\xivp{\phi})=n^{-1} \sum_{i=1}^n \log p_\thv(\x_i)- D \nell(\xivp{\phi}).
    \end{equation*}
    The first term of the rhs converges $\PP_\star-\as$ to $\PE_\star\lrb{\log p_\thv(\x)}$ according to the strong Law of Large Numbers. We now prove that $\PP_\star-\as$, $\lim D \nell(\xivp{\phi})= 0$ as $n\wedge \lo \to \infty$. By \eqref{eq:rnk:negatif}, $D \nell(\xivp{\phi})\geq 0$. It thus remains to show that $\PP_\star-\as$, $\limsup D \nell(\xivp{\phi})\leq 0$ as $n\wedge \lo \to \infty$. The rest of the argument is devoted to obtaining this limiting result. 

    Fix an arbitrary integer $m$. For any integer $\lo$, write the Euclidean division $\lo=m \e{\lo/m}+r$ where $r\in [0:m-1]$. Then, setting $\bar S_\lo(\bz,\x)\eqdef \lo^{-1} \sum_{\ell=1}^\lo r_\xivp{\phi}^\x(\bz^\ell)$, we obviously have 
    \begin{equation} \label{eq:decomp:block}
        \bar S_\lo(\bz,\x)= \frac 1 \lo \lr{\sum_{\j=0}^{\e{\lo/m}-1} m \ a^m_\j(\bz,\x) + \sum_{\ell=1}^r b_l(\bz,\x)},  
    \end{equation}
    where $a^m_\j(\bz,\x)= m^{-1}\sum_{\ell=m \j+1}^{m(\j +1)}  r_\xivp{\phi}^\x(\bz^\ell)$ and $b_l(\bz,\x)=r_\xivp{\phi}^\x(\bz^{m \e{\lo/m}+l})$. The rhs of \eqref{eq:decomp:block} is an average with uneven weights and the concavity of the log function then yields for all $i \in [1:n]$, 
    $$ 
    \log \bar S_\lo(\bz_i,\x_i) \geq \frac 1 \lo \lr{\sum_{\j=0}^{\e{\lo/m}-1} m \ \log a^m_\j(\bz_i,\x_i) + \sum_{\ell=1}^r \log b_l(\bz_i,\x_i)}.
    $$
    Taking the conditional expectation wrt $\X$ and then the average wrt the index $i$,      
    \begin{align*}
        -D \nell(\xivp{\phi})&\geq  \frac 1 n \sum_{i=1}^n \frac 1 \lo \lr{\sum_{\j=0}^{\e{\lo/m}-1} m \ \PE^\phi_\X\lrb{\log a^m_\j(\bz_i,\x_i)} + \sum_{\ell=1}^r \PE^\phi_\X\lrb{\log b_l(\bz_i,\x_i)}}     \\
        & = \frac{m \e{\lo/m}}{ \lo} \ \lr{\frac 1 n \sum_{i=1}^n \PE^\phi_\X\lrb{\log a^m_0(\bz_i,\x_i)}} + \frac r \lo \lr{\frac 1 n \sum_{i=1}^n \PE^\phi_\X\lrb{\log b_1(\bz_i,\x_i)}   },
    \end{align*}
    where in the last equality, we have used that neither $\PE^\phi_\X \lrb{\log a^m_\j(\bz_i,\x_i)}$ depends on $\j$, nor $\PE^\phi_\X\lrb{\log b_l(\bz_i,\x_i)}$ on $k$. Letting now $n \wedge \lo$ go to infinity, we finally obtain that $\PP_\star-\as$, 
   $$
   -\limsup_{n\wedge \lo \to \infty} D \nell (\xivp{\phi})=\liminf_{n\wedge \lo \to \infty} -D \nell (\xivp{\phi}) \geq \PE^\phi_\X \lrb{\log \lr{\frac 1 m \sum_{\ell=1}^m  r_\xivp{\phi}^\x(\bz^\ell) }}.
   $$
Note that the integer $m$ in the rhs is arbitrary. Now, \Cref{lem:consist:log} applied to $U_\ell=r_\xivp{\phi}^\x(\bz^\ell)$ and $\PE=\PE_\star \PE^\phi_\X$ yields
   $$ 
    \lim_{m \to\infty} \PE_\star \PE^\phi_\X \lrb{\log \lr{\frac 1 m \sum_{\ell=1}^m  r_\xivp{\phi}^\x(\bz^\ell) }}=0.
   $$
   Hence, $-\limsup_{n\wedge \lo \to \infty} D \nell (\xivp{\phi}) \geq 0$, $\PP_\star-\as$ and the proof is concluded.  
\end{proof}

We are now ready to prove Theorem \ref{thm:consistency}.
\begin{proof}

    Let $\Wset$ be any open neighborhood of $\thv$. We will show that there exists a $\PP_\star-\as$ finite random integer $m$ such that $\t \theta \nell \in \Wset$ for $n\wedge \lo \geq m$.  Write $\bar \Wset$ the compact set $\bar \Wset=\Theta \setminus \Wset$ and for any $\theta \in \Theta$, define $\l(\theta)=\PE_\star[\log p_\theta(\x)]$. By \ref{hyp:theta:star}, we have for any $\theta_0 \in \bar \Wset$, 
    $$ 
     \l(\theta_0) < \l(\thv).
    $$
    Since $ \rho \mapsto f_\rho(\x)=\sup_{\theta\in \Theta} \log^+ p_\theta(\x)- \sup_{\theta \in \ball{\theta_0,\rho}} \log p_\theta(\x)\geq 0$ and non-decreasing when $\rho$ decreases to $0$, monotone convergence yields: $\lim_{\rho \searrow 0} \PE_\star[f_\rho(\x)]= \PE_\star[ \lim_{\rho \searrow 0} f_\rho(\x)]$. Expanding the expression of $f_\rho$ and using \ref{hyp:unif}-(ii) thus show that 
    $$ 
     \lim_{\rho \searrow 0} \PE_\star\lrb{ \sup_{\theta \in \ball{\theta_0,\rho}} \log p_\theta(\x) } = \PE_\star\lrb{\lim_{\rho \searrow 0}  \lr{\sup_{\theta \in \ball{\theta_0,\rho}} \log p_\theta(\x)} } \leq  \PE_\star[\log p_{\theta_0}(\x)]=\l(\theta_0),
    $$ 
    where the last inequality follows from \ref{hyp:unif}-(i). Hence, for any $\theta_0\in \bar \Wset$, there exists $\rho_0>0$ such that  
    \begin{equation*}
        \PE_\star\lrb{ \sup_{\theta \in \ball{\theta_0,\rho_0}} \log p_\theta(\x) } < \l(\thv) .
    \end{equation*}
    By compactness of $\bar \Wset$, there exist $\ell\geq 1$ and $\set{(\theta_j,\rho_j) \in \bar \Wset \times \rsetpos}{j \in [1:\ell]}$ such that 
    \begin{equation}
        \label{eq:consist:one}
        \bar \Wset \subset \cup_{j=1}^\ell \ball{\theta_j,\rho_j} \quad \mbox{and}  \quad \PE_\star\lrb{ \sup_{\theta \in \ball{\theta_j,\rho_j}} \log p_\theta(\x) } < \l(\thv) \eqsp, \quad j \in [1:\ell].
    \end{equation}
    Using \eqref{eq:minor:elbo}, the strong Law of Large Numbers and \eqref{eq:consist:one}, we have $\PP_\star-\as$, 
    \begin{align}
        \label{eq:consist:two}
        \limsup_{n\wedge \lo \to \infty} \sup_{\xi \in \bar \Wset \times \Phi}& n^{-1} \lk  \nell (\xi)  \leq \limsup_{n \to \infty}  \sup_{\theta \in \bar\Wset} n^{-1} \sum_{i=1}^n \log p_\theta(\x_i) =\limsup_{n \to \infty}  \sup_{j\in[1:\ell]}\sup_{\theta \in \ball{\theta_j,\rho_j}} n^{-1} \sum_{i=1}^n \log p_\theta(\x_i) \nonumber \\
        &\leq \limsup_{n \to \infty}  \sup_{j\in[1:\ell]}n^{-1} \sum_{i=1}^n \sup_{\theta \in \ball{\theta_j,\rho_j}} \log p_\theta(\x_i) = \sup_{j\in[1:\ell]}  \limsup_{n \to \infty}  n^{-1} \sum_{i=1}^n \sup_{\theta \in \ball{\theta_j,\rho_j}} \log p_\theta(\x_i) \\
        &= \sup_{j\in[1:\ell]} \PE_\star\lrb{ \sup_{\theta \in \ball{\theta_j,\rho_j}} \log p_\theta(\x) } <\l(\thv)\,.
    \end{align}
But for any  $\phi\in\Phi$, applying \Cref{lem:consistency:one}, we have $\PP_\star-\as$, $\l(\thv)=\lim_{n\wedge \lo\to \infty}
n^{-1} \lk  \nell (\xivp{\phi})$ where we used the notation
$\xivp{\phi}=(\thv,\phi)$. Finally, there exists a $\PP_\star-\as$ finite random integer $m$ such that, provided that $n\wedge \lo\geq m$, we have 
    $$ 
    \sup_{\xi \in \bar \Wset \times \Phi} n^{-1} \lk  \nell (\xi) < n^{-1} \lk  \nell (\xivp{\phi}) \leq n^{-1} \lk  \nell (\t \xi  \nell)\eqsp, \quad \PP_\star-\as
    $$ 
    This implies $\t \xi  \nell \notin \bar \Wset \times \Phi$, hence $\t \theta \nell \in \Wset$  and the proof is completed. 
\end{proof}

\section{Proof of Theorem \ref{thm:consistency-variational}}
\label{proof:thm:variational}

We first need the following lemma that holds under Assumption \ref{hyp:moment} only, and that mainly states that the variational gap $n D \nell(\xi)$ is asymptotically of order $\Vratio(\xi)/{2\lo}$, whatever the value of $\xi=(\theta,\phi)\in\Xi$. Notice that $\Vratio(\xi)/{2\lo}$ is exactly the first-order term in the development \eqref{devIWELBO}.

\begin{lemma}
    \label{lem:ell-D-lim}
        Assume \ref{hyp:moment}. Then, 
        \begin{enumerate}[(i)]
            \item \label{item::ell-D-lim:one}for every  $\xi=(\theta,\phi)\in\Xi$, we have $\PP_\star-\as$,
            $$ 
            \lim_{n\wedge \lo \to \infty} \lo D \nell(\xi) =\Vratio(\xi)/2.  
            $$
            \item \label{item::ell-D-lim:two} Moreover, for every set $\Cset \subset \Xi$, we have $\PP_\star-\as$, 
\begin{align}
    &             \liminf_{n\wedge \lo \to \infty} \lr{\inf_{\xi \in \Cset}\lo D \nell(\xi)} \geq \PE_\star \lrb{\inf_{\xi \in \Cset}  \Vratio_\xi(\x)},\label{eq:inf:D}\\
    &             \limsup_{n\wedge \lo \to \infty} \lr{\sup_{\xi \in \Cset}\lo D \nell(\xi)} \leq \PE_\star \lrb{\sup_{\xi \in \Cset}  \Vratio_\xi(\x)}.\label{eq:sup:D}
\end{align}    
        \end{enumerate}
\end{lemma}
\begin{proof}
\begin{enumerate}[(i)]
    \item 
For any $\xi=(\theta,\phi) \in \Xi$, set $\bar S_\lo(\bz,\x)\eqdef \lo^{-1} \sum_{\ell=1}^\lo r_\xi^\x(\bz^\ell)$ and   
\begin{equation*}
    A_\lo (\x)\eqdef \lo\PE^\phi_\X \lrb{\log \lr{\bar S_\lo(\bz,\x)}}\,,\quad \Vratio_\xi (\x)\eqdef\PE^\phi_\X \lrb{\lr{ r_\xi(\bz^1,\x)-1}^2}\,, \quad C_\lo (\x)=A_\lo (\x)+\Vratio_\xi (\x)/2.
\end{equation*}
Applying \Cref{lem:log} to $r=\bar S_\lo(\bz,\x)$ and noting that $\PE^\phi_\X [\bar S_\lo(\bz,\x)]=1$ and $\lo\PE^\phi_\X \lrb{\lr{\bar S_\lo(\bz,\x)-1}^2}=\Vratio_\xi(\x)$, we finally get the bound
\begin{equation} \label{eq:bound-C-ell-one}
    |C_\lo(\x)|  \leq  \lo \PE^\phi_\X\lrb{ (\bar S_\lo(\bz,\x)-1)^2 \times \left| \log(\bar S_\lo(\bz,\x))\right|}.
\end{equation}
Now, write the decomposition
\begin{equation}
    -\lo D \nell(\xi) = n^{-1} \sum_{i=1}^n A_\lo(\x_i)=- {n^{-1} \sum_{i=1}^n \Vratio_\xi(\x_i)} /2 + n^{-1} \sum_{i=1}^n C_\lo(\x_i). \label{eq:bound:remainder}  
\end{equation}
The first term of the rhs of \eqref{eq:bound:remainder} converges $\PP_\star-\as$ to
$-\Vratio(\xi)/2$ by the strong LLN. We now show that $\PP_\star-\as$,
$\lim n^{-1} \sum_{i=1}^n C_\lo(\x_i)=0$ as $n\wedge \lo$ tends to
infinity. Applying the Cauchy-Schwarz inequality to
\eqref{eq:bound-C-ell-one} yields 
\begin{align*}
    |C_\lo(\x)| & \leq   \lo \normsx{\lr{\bar
                   S_\lo(\bz,\x)-1}^2}{\phi}{2} \normsx{\log(\bar S_\lo(\bz,\x))}{\phi}{2} =\lo \normsx{\bar S_\lo(\bz,\x)-1}{\phi}{4}^2 \normsx{\log(\bar S_\lo(\bz,\x))}{\phi}{2}.
\end{align*}
We now bound the first term of the rhs by Marcinkiewicz-Zygmund's inequality and use \Cref{lem:borne:log:U}-(ii) to bound the second one. Then, there exists a constant $M$ not depending on $(\lo,\x,\xi)$ such that 
\begin{align} \label{eq:bound:Cell}
    |C_\lo(\x)| & \leq   M \lo^{-1/2}\normsx{r_\xi^\x(\bz)}{\phi}{4}^3 \lr{\normsx{\log(r_\xi^\x(\bz))}{\phi}{4}+1}.
\end{align}
Then,  
\begin{align*}
0\leq  n^{-1}\sum_{i=1}^n |C_\lo(\x_i)| &\leq M \lo^{-1/2} n^{-1}\sum_{i=1}^n \normsx{r_\xi^{\x_i}(\bz)}{\phi}{4}^3 \lr{\normsx{\log(r_\xi^{\x_i}(\bz))}{\phi}{4}+1} ,
\end{align*}
which converges $\PP_\star-\as$ to $0$ as $n \wedge \lo \to \infty$ under \ref{hyp:moment}. The proof of (i) is completed.  
\item Noting that   
\begin{equation*}
    \sup_{\xi\in\Cset} n^{-1}\sum_{i=1}^n |C_\lo(\x_i)| \leq n^{-1} \sum_{i=1}^n \sup_{\xi\in\Cset}  |C_\lo(\x_i)|, 
\end{equation*} 
and using \eqref{eq:bound:Cell} to bound the rhs we deduce that under \ref{hyp:moment},  $\PP_\star-\as$, $ \sup_{\xi\in\Cset} n^{-1} \sum_{i=1}^n |C_\lo(\x_i)|$ converges to $0$ as $n\wedge \lo \to \infty$. 
Considering \eqref{eq:bound:remainder}, we finally have, $\PP_\star-\as$, 
\begin{align*}
   \limsup_{n\to\infty} \lr{\sup_{\xi\in\Cset} -\lo D \nell(\xi)} &\leq \limsup_{n\to\infty} \lr{\sup_{\xi\in\Cset} - {n^{-1} \sum_{i=1}^n \Vratio_\xi(\x_i)} /2} + \limsup_{n\to\infty} \lr{\sup_{\xi\in\Cset} n^{-1} \sum_{i=1}^n C_\lo(\x_i)} \\
   &= - \liminf_{n\to\infty} \lr{\inf_{\xi\in\Cset} {n^{-1} \sum_{i=1}^n \Vratio_\xi(\x_i)} /2} \leq  - \liminf_{n\to\infty}  {n^{-1} \sum_{i=1}^n \inf_{\xi\in\Cset} \Vratio_\xi(\x_i)} /2\\
   &=-\PE_\star\lrb{\inf_{\xi \in \Cset} \Vratio_\xi(\x) }/2,
\end{align*}
which shows \eqref{eq:inf:D}. Similarly \eqref{eq:sup:D} may be proved by following the same lines as the proof of \eqref{eq:inf:D} and replacing $\liminf_{n\to\infty}$ by $\limsup_{n\to\infty}$ and $\inf_{\xi \in \Cset}$ by $\sup_{\xi \in \Cset}$. It is omitted for brevity.  
\end{enumerate}
\end{proof}

We are now ready to prove Theorem \ref{thm:consistency-variational}.

\begin{proof}
Let $\Wset$ be any open neighborhood of $\phiv$. We will show that there exists a $\PP_\star-\as$ finite random integer $m$ such that $\t \phi \nell \in \Wset$ as soon a $n\wedge \lo \geq m$.  Write $\bar \Wset$ the compact set $\bar \Wset=\Phi \setminus \Wset$ and recall that $\Vratio_\xi(\x)\eqdef\PE^\phi_\X \lrb{
    \lr{r_\xi^\x(\bz)-1}^2}$ so that $\Vratio(\theta,\phi)=\PE_\star\lrb{\Vratio_\xi(\x)}$ for any $\xi=(\theta,\phi) \in \Xi$. By \ref{hyp:variational}-(i), we have for any $\phi_0 \in \bar \Wset$, 
$$ 
 \Vratio(\thv,\phi_0) > \Vratio(\thv,\phiv).
$$
For any $\xi=(\theta,\phi) \in\Xi$, define the notation $\tball{\xi,\rho}=\ball{\theta,\rho} \times \ball{\phi,\rho}$. When $\rho$ is decreasing, the function $\rho \mapsto \inf_{\xi \in \tball{\xivp{\phi_0},\rho} } \Vratio_{\xi}(\x)$ is non-decreasing and non-negative. The monotone convergence then yields
$$ 
 \lim_{\rho \searrow 0} \PE_\star\lrb{ \inf_{\xi \in \tball{\xivp{\phi_0},\rho}}  \Vratio_{\xi}(\x) } = \PE_\star\lrb{\lim_{\rho \searrow 0}  \lr{\inf_{\xi \in \tball{\xivp{\phi_0},\rho}} \Vratio_{\xi}(\x)} } =\Vratio(\thv,\phi_0),
$$
where the last equality follows from the continuity of $\xi \mapsto \Vratio_\xi(\x)$ as stated in \ref{hyp:variational}-(ii). 
Hence, for any $\phi_0\in \bar \Wset$, there exists $\rho_0>0$ such that  
\begin{equation*}
    \PE_\star\lrb{ \inf_{\xi \in \tball{\xivp{\phi_0},\rho_0}}  \Vratio_{\xi}(\x) } > \Vratio(\thv,\phiv).
\end{equation*}
By compactness of $\{\thv\}\times \bar \Wset$, there exist $\gamma>0$, $\ell\geq 1$ and $\set{(\phi_j,\rho_j) \in \bar \Wset \times \rsetpos}{j \in [1:\ell]}$ such that 
\begin{equation*}
    \bar \Wset \subset \cup_{j=1}^\ell \ball{\phi_j,\rho_j} \quad \mbox{and}  \quad \PE_\star\lrb{ \inf_{\xi \in \tball{\xivp{\phi_j},\rho_j}}  \Vratio_{\xi}(\x) } > \Vratio(\thv,\phiv)+\gamma \eqsp, \quad j \in [1:\ell].
\end{equation*}
Denote $\check \rho=\min\set{\rho_j}{j \in [1:\ell]}$. Then, applying \Cref{lem:ell-D-lim}-\ref{item::ell-D-lim:two}, we have
$\PP_\star-\as$, 
\begin{align}
    \label{eq:consist:var:two}
    \liminf_{n\wedge \lo \to \infty} \lr{\inf_{\xi \in \ball{\thv,\check \rho} \times \bar \Wset } \lo D \nell(\xi)} &\geq \inf_{j\in [1:\ell]}  \liminf_{n\wedge \lo \to \infty} \lr{\inf_{\xi \in \tball{\xivp{\phi_j},\rho_j}}\lo D \nell(\xi)} \nonumber \\
    &\geq \inf_{j\in [1:\ell]} \PE_\star\lrb{ \inf_{\xi \in \tball{\xivp{\phi_j},\rho_j}}  \Vratio_{\xi}(\x) } > \Vratio(\thv,\phiv)+\gamma.
\end{align}
The constant $\gamma>0$ being fixed, we now show that there exists a sufficiently small $\rho'>0$ such that 
\begin{equation}
    \label{eq:V:upper-bound}
    \PE_\star\lrb{\sup_{\theta\in \ball{\thv,\rho'}} \Vratio_{(\theta,\phiv)}(\x)} < \Vratio(\thv,\phiv)+\gamma/2.
\end{equation}
Indeed, when $\rho'$ decreases, the function 
$$
\rho'\mapsto g_{\rho'}(\x)=\sup_{\xi \in \Xi}\PE^\phi_\X\lrb{\lrcb{r_\xi^\x(\bz)}^2}- \sup_{\theta\in \ball{\thv,\rho'}} \Vratio_{(\theta,\phiv)}(\x)
$$
is non-negative and non-increasing. The monotone convergence theorem yields 
$$
\lim_{\rho'\searrow 0} \PE_\star[g_{\rho'}(\x)]= \PE_\star[\lim_{\rho'\searrow 0} g_{\rho'}(\x)]\eqsp.
$$
 Expanding the expression of $g$ and combining with \ref{hyp:moment}, we finally obtain 
$$ 
\lim_{\rho'\searrow 0}\PE_\star\lrb{\sup_{\theta\in \ball{\thv,\rho'}} \Vratio_{(\theta,\phiv)}(\x)}=\PE_\star\lrb{\lim_{\rho'\searrow 0} \sup_{\theta\in \ball{\thv,\rho'}} \Vratio_{(\theta,\phiv)}(\x)}  =\Vratio(\thv,\phiv),
$$
where the last equality follows from continuity assumption \ref{hyp:variational}-(ii). Hence, \eqref{eq:V:upper-bound} is shown. Then, for $\rho'$ chosen as in \eqref{eq:V:upper-bound}, we can apply again \Cref{lem:ell-D-lim}-\ref{item::ell-D-lim:two} and obtain  
$\PP_\star-\as$, 
\begin{equation}
    \label{eq:consist:var:three}
    \limsup_{n\wedge \lo \to \infty} \lr{\sup_{\theta \in \ball{\thv,\rho'}  } \lo D \nell(\theta,\phiv)} \leq \PE_\star\lrb{\sup_{\theta\in \ball{\thv,\rho'}} \Vratio_{(\theta,\phiv)}(\x)} < \Vratio(\thv,\phiv)+\gamma/2.
\end{equation}
Now, note that by \Cref{thm:consistency}, $\PP_\star-\as$, $\theta \nell \to \thv$ as $n\wedge \lo \to\infty$. Combining this with \eqref{eq:consist:var:two} and \eqref{eq:consist:var:three}, there exists a $\PP_\star-\as$ finite random integer $m$ such that for all $n\wedge k \geq m$,   
\begin{equation}
    \label{eq:consist:var:four}
    \t \theta \nell \in\ball{\thv,\check \rho \wedge \rho'}\,,\quad \mbox{and} \quad \sup_{\theta \in \ball{\thv,\rho'}  } \lo D \nell(\theta,\phiv) < \inf_{\xi \in \ball{\thv,\check \rho} \times \bar \Wset } \lo D \nell(\xi).
\end{equation}
We now show by contradiction that $\t \phi \nell\in\Wset$. Indeed assume that $\t \phi \nell\in \bar \Wset$, then by \eqref{eq:consist:var:four}, $\t \xi \nell=(\t \theta \nell,\t \phi \nell) \in \ball{\thv,\check \rho} \times \bar \Wset$, which in turn implies, via \eqref{eq:def:xi:n:ell}, 
\begin{align*}
    & 0\leq  \lo\lr{\lk  \nell(\t \xi \nell)-\lk  \nell(\tilde \theta \nell, \phiv)}=-\lo D \nell(\t \xi \nell)+\lo D \nell(\tilde \theta \nell, \phiv) \leq  -\inf_{\xi \in \ball{\thv,\check \rho} \times \bar \Wset } \lo D \nell(\xi) + \sup_{\theta \in \ball{\thv,\rho'}  } \lo D \nell(\theta,\phiv).
\end{align*}
This is in contradiction with \eqref{eq:consist:var:four}. Hence $\t \phi \nell\in\Wset$ as soon as $n\wedge \lo\geq m$ and the proof is completed. 

\end{proof}

\section{Proof of Theorem~\ref{thm:AS}}
\label{proof:thm:efficiency}
The two following propositions will be used in the proof of
\Cref{thm:AS} and proved afterwards. 

\begin{proposition}
    \label{prop:weakconv}
    Under the assumptions of \Cref{thm:AS}, we have 
    $$
    n^{-1/2}\nabla_\xi \lk \nelln (\xi^\star) \weakconv_{\PP_\star} \mathcal{N}\lr{0,\begin{pmatrix} 
  J_{1}&0\\
  0&0
  \end{pmatrix}}.
  $$
\end{proposition}
  
\begin{proposition} \label{prop:second:derivative}
    Under the assumptions of \Cref{thm:AS}, we have  
    $$ 
     \sup_{\xi=(\theta,\phi) \in \Xi} \norm{n^{-1}\nabla^2_\xi \lk_n^{\lo_n}(\xi) - \begin{pmatrix}
        n^{-1}\sum_{i=1}^n \nabla^2_\theta \log p_\theta(\x_i)&0\\
        0&0 \end{pmatrix}       } \convprob{\PP_\star}0.
    $$
  \end{proposition}

\begin{proof}[(Proof of \Cref{thm:AS})]
Since $\nabla_\xi \lk \nelln (\t \xi \nelln)=0$, we have 
\begin{align*}
    -n^{-1/2}\nabla_\xi \lk \nelln (\xi^\star)&=n^{-1/2}\nabla_\xi \lk \nelln (\t \xi \nelln) -n^{-1/2}\nabla_\xi \lk \nelln (\xi^\star) \\
& =U_n \lrb{n^{1/2}(\xi \nelln-\xi^\star)} ,
\end{align*}
where $U_n=n^{-1}\int_0^1  \nabla_\xi^2 \lk \nelln (t \t \xi \nelln+(1-t) \xi^\star)\rmd t$ 
and the proof is completed by applying Slutsky's lemma, provided we show that 
\begin{align}
   & n^{-1/2}\nabla_\xi \lk \nelln (\xi^\star) \weakconv_{\PP_\star} \mathcal{N}\lr{0,\begin{pmatrix} 
        J_{1}&0\\
        0&0
        \end{pmatrix}}, \label{eq:as:weakconv:one} \\
    &  U_n \convprob{\PP_\star} \begin{pmatrix} 
            J_{2}&0\\
            0&0
            \end{pmatrix}. \label{eq:as:weakconv:two}
\end{align}
Applying \Cref{prop:weakconv}, we get \eqref{eq:as:weakconv:one}. We now turn to the proof of \eqref{eq:as:weakconv:two}. To this aim, let $\gamma,\epsilon>0$.  
Under \ref{hyp:as:zero}, the dominated convergence theorem shows that 
$$ 
\lim_{\rho \to 0}\PE \lrb{ \sup_{\theta \in \ball{\thv,\rho}} \norm{\nabla^2_\theta \log p_\theta(\x_i)-\nabla^2_\theta \log p_\thv(\x_i)} } = 0.
$$ 
Therefore, there exists $\rho>0$ small enough so that 
\begin{equation} \label{eq:controle:moment}
    \PE \lrb{ \sup_{\theta \in \ball{\thv,\rho}} \norm{\nabla^2_\theta \log p_\theta(\x_i)-\nabla^2_\theta \log p_\thv(\x_i)} } \leq  \epsilon \gamma  / 3.
\end{equation}
This $ \rho $ being chosen, \ref{hyp:consist} implies that there exists $n_0 \in \nset$ such that for any $n\geq n_0$, 
\begin{equation*} 
    \PP_\star \lr {\t \xi \nelln \notin \tball{\xiv,\rho}} \leq \delta,
\end{equation*} 
where  $\tball{\xi,\rho}=\ball{\theta,\rho} \times \ball{\phi,\rho}$. Now, by the triangular inequality, 
\begin{align*}
    &\PP_\star\lr{ \norm{U_n-\begin{pmatrix}
     J_{2}&0\\
    0&0 \end{pmatrix}  } \geq \epsilon\, , \quad \t \xi \nelln \in \tball{\xiv,\rho}} \\
    &\quad \leq 
     \PP_\star\lr{\sup_{\xi=(\theta,\phi) \in \Xi} \norm{n^{-1}\nabla^2_\xi \lk_n^{\lo_n}(\xi) - \begin{pmatrix}
        n^{-1}\sum_{i=1}^n \nabla^2_\theta \log p_\theta(\x_i)&0\\
        0&0 \end{pmatrix}       } \geq \epsilon/3}\\
        & \quad \quad + \PP_\star\lr{  n^{-1}\sum_{i=1}^n \sup_{\theta \in \ball{\thv,\rho}} \norm{\nabla^2_\theta \log p_\theta(\x_i)-\nabla^2_\theta \log p_\thv(\x_i)} \geq \epsilon/3} \\
        & \quad \quad \quad + \PP_\star\lr{   \norm{ n^{-1}\sum_{i=1}^n \nabla^2_\theta \log p_\thv(\x_i)-\PE_\star\lrb{\nabla^2_\theta \log p_\thv(\x)}} \geq \epsilon/3}.
\end{align*}
The second term of the rhs can be bounded by $\delta$ for any $n\geq 1$ by Markov's inequality combined with \eqref{eq:controle:moment}. Then, applying \Cref{prop:second:derivative} for the first term of the rhs and the law of large numbers under \ref{hyp:as:zero} for the third term of the rhs, there exists $n_1 \geq n_0$ such that the first and the last terms of the rhs are both less than $\delta$ for any $n \geq n_1$. Finally, for all $n \geq n_1$, we get 
\begin{align*}
    \PP_\star\lr{ \norm{U_n-\begin{pmatrix}
        J_{2}&0\\
        0&0 \end{pmatrix}  } \geq \epsilon} &\leq \PP_\star \lr {\t \xi \nelln \notin \tball{\xiv,\rho}} + \PP_\star\lr{ \norm{U_n-\begin{pmatrix}
            J_{2}&0\\
            0&0 \end{pmatrix}  } \geq \epsilon\, , \quad \t \xi \nelln \in \tball{\xiv,\rho}}\\
            & \leq 4\delta.
\end{align*}
and the proof is completed. 
\end{proof}

\begin{proof}[(Proof of \Cref{prop:weakconv})]
    Differentiating \eqref{eq:def:L}, we get 
    \begin{align*}
      n^{-1/2}\nabla_\xi \lk \nelln (\xi^\star)=n^{-1/2}\sum_{i=1}^n \nabla_\xi \log p_\thv(\x_i)+ A_n \quad \mbox{where} \quad A_n=n^{-1/2} \sum_{i=1}^n \nabla_\xi \PE_\X \lrb{  \log \frac 1 {\lo_n} \sum_{\ell=1}^{\lo_n} \frac{f_\xiv(\epsilon_\ell|\x_i)}{\nu(\epsilon_\ell)}}.
  \end{align*}
  Under \ref{hyp:as:zero}, the Central Limit Theorem holds and  
  $$ 
  n^{-1/2}\sum_{i=1}^n \nabla_\xi \log p_\thv(\x_i)=\begin{pmatrix} 
      n^{-1/2}\sum_{i=1}^n \nabla_\theta \log p_\thv(\x_i)\\
      0
  \end{pmatrix}  \weakconv \begin{pmatrix} 
  \mathcal{N}(0,J_{1})\\
  0
  \end{pmatrix}.
  $$
  To complete the proof, we must show $A_n \convprob{\PP}0$ as $n  \to \infty$. To this aim, we will prove that $\PE\lrb{\norm{A_n}} \rightarrow 0$ and then apply the Markov inequality to get $A_n \convprob{\PP} 0$ as $n  \to \infty$. Setting $B  \nelln =n^{-1/2}A_n$, we have: 
  $$ 
    B_n= \frac 1 n \sum_{i=1}^n \PE_\X \lrb{ C_n(\beps,\x_i)},
  $$
  where 
  \begin{align*}
      & C_n(\beps,\x)=\frac{\sum_{\ell=1}^{\lo_n} \varphi(\epsilon_\ell,\x)}{\sum_{\ell=1}^{\lo_n} a(\epsilon_\ell,\x)} \,, \quad a(\epsilon,\x)=\frac{f_\xiv(\epsilon|\x)}{\nu(\epsilon)}\,, \quad \varphi(\epsilon,\x)=\frac{\nabla_\xi f_\xiv(\epsilon|\x)} {\nu(\epsilon)}.
  \end{align*}
  Since 
  \begin{equation} \label{eq:boundB}
      \PE\lrb{\norm{B_n}} \leq D_n\eqdef \PE_\star \lrb{\norm{\PE_\X \lrb{C_n(\beps,\x)}}},
  \end{equation}
  we only need to bound the rhs $D_n$ (which only depends on $n$
  through $\lo_n$). Applying \Cref{lem:borne:autonorm}, we obtain: for every $\alpha>2$,  there exists a constant $M$ such that 
      \begin{equation*}
      \norm{\PE_\X[C_{k_n}(\beps,\x)]} \leq   M k_n^{-1}\lrb{ k_n^{1/\alpha} M_{0,\xi^\star}(\x) + N_{0,\xi^\star}(\x)}.
  \end{equation*}
  Finally, recalling \eqref{eq:boundB} and $\delta=\alpha/(\alpha-1)$, there exists a finite constant $M$ such that for any $n$,  
  \begin{align} \label{eq:A}
      \PE\lrb{\norm{A_n}}= n^{1/2} \PE\lrb{\norm{B_n}}
      &\leq n^{1/2} \PE_\star \lrb{\norm{\PE_\X \lrb{C_n(\beps,\x)}}} \\
      &\leq M n^{1/2} {\lo_n}^{-1/\delta} \PE_\star\lrb{M_{0,\xi^\star}(\x)}+ M  n^{1/2} \lo_n^{-1} \PE_\star \lrb{N_{0,\xi^\star}
  (\x)} 
    \end{align}
  under \ref{hyp:weak:one}. Note that since $\lim_{n \to \infty}{\lo_n}/n^{\delta/2}=\infty$, the rhs of \eqref{eq:A} tends to $0$. Then, $A_n \convprob{\PP}0$ and the proof is completed. 
    
\end{proof}

\begin{proof}[(Proof of \Cref{prop:second:derivative})]
Setting 
\begin{align*}
    B_n(\xi)&\eqdef n^{-1}\sum_{i=1}^n \nabla^2_\xi \PE_\X \lrb{\log
    \lr{\frac 1 {\lo_n} \sum_{\ell=1}^{\lo_n}
    \frac{f_\xi(\epsilon_\ell|\x_i)}{\nu(\epsilon_\ell)}}}=n^{-1}\sum_{i=1}^n
    \PE_\X \lrb{\nabla^2_\xi \log \lr{\frac 1 {\lo_n} \sum_{\ell=1}^{\lo_n} \frac{f_\xi(\epsilon_\ell|\x_i)}{\nu(\epsilon_\ell)}}}, 
    \end{align*}
we can differentiate twice \eqref{eq:def:L} and obtain  
    \begin{align*}
       n^{-1} \nabla_\xi^2 \lk_n^{\lo_n}(\xi)&= n^{-1} \sum_{i=1}^n \nabla_\xi^2 \log p_\theta(\x_i)+ B_n(\xi)\\
       &=\begin{pmatrix}
        n^{-1}\sum_{i=1}^n \nabla^2_\theta \log p_\theta(\x_i)&0\\
        0&0 \end{pmatrix} + B_n(\xi).
    \end{align*}
Then,  the proof is completed provided we show
\begin{equation} \label{eq:convUnifB}
    \sup_{\xi \in \Xi} \norm{B_n(\xi)} \convprob{\PP_\star} 0.
\end{equation}
The rest of the proof is devoted to establishing \eqref{eq:convUnifB}. Rewrite $B_n(\xi)$ as 
$$ 
B_n(\xi)=n^{-1} \sum_{i=1}^n \PE_{\X}\lrb{\frac{\sum_{\ell=1}^{\lo_n} \psi_\xi(\epsilon_\ell,\x_i)}{\sum_{\ell=1}^{\lo_n} a_\xi(\epsilon_\ell,\x_i)}} + n^{-1} \sum_{i=1}^n \PE_{\X}\lrb{\lr{\frac{\sum_{\ell=1}^{\lo_n} \varphi_\xi(\epsilon_\ell,\x_i)}{\sum_{\ell=1}^{\lo_n} a_\xi(\epsilon_\ell,\x_i)}} \lr{\frac{\sum_{\ell=1}^{\lo_n} \varphi_\xi(\epsilon_\ell,\x_i)}{\sum_{\ell=1}^{\lo_n} a_\xi(\epsilon_\ell,\x_i)}}^T},
$$ 
where $\psi_\xi(\epsilon,\x_i)= \frac{\nabla_\xi^2 f_\xi(\epsilon|\x_i)}{\nu(\epsilon)}$, $a_\xi(\epsilon,\x_i)= \frac{ f_\xi(\epsilon|\x_i)}{\nu(\epsilon)}$,  and $\varphi_\xi(\epsilon,\x_i)= \frac{\nabla_\xi f_\xi(\epsilon|\x_i)}{\nu(\epsilon)}$. Then, using Holder's inequality, 
\begin{equation} \label{eq:holder}
\norm{B_n(\xi)} \leq n^{-1} \sum_{i=1}^n \norm{\PE_{\X}\lrb{\frac{\sum_{\ell=1}^{\lo_n} \psi_\xi(\epsilon_\ell,\x_i)}{\sum_{\ell=1}^{\lo_n} a_\xi(\epsilon_\ell,\x_i)}}} + n^{-1} \sum_{i=1}^n 
\PE_{\X}\lrb{\norm{\frac{\sum_{\ell=1}^{\lo_n} \varphi_\xi(\epsilon_\ell,\x_i)}{\sum_{\ell=1}^{\lo_n} a_\xi(\epsilon_\ell,\x_i)}}^2 }.
\end{equation}
Since $\PE_{\X}[\psi_\xi(\epsilon,\x)]=\PE_{\X}[\varphi_\xi(\epsilon,\x)]=0$ and $\PE_{\X}[a_\xi(\epsilon,\x)]=1$, we can bound the two terms of the rhs by applying first \Cref{lem:borne:autonorm}-\ref{item:autonorm:three} and then \Cref{lem:borne:autonorm}-\ref{item:autonorm:two} and there exists a constant $M$ such that  
\begin{align*}
    \sup_{\xi \in\Xi} \norm{B_n(\xi)} &\leq M k_n^{-1/\delta} \lr{ n^{-1} \sum_{i=1}^n \sup_{\xi \in\Xi} M_{1,\xi}(\x_i)}+M k_n^{-1} \lr{ n^{-1} \sum_{i=1}^n \sup_{\xi \in\Xi} N_{1,\xi}(\x_i)}\\
& \quad + M k_n^{-1+2/\alpha}    \lr{n^{-1} \sum_{i=1}^n\sup_{\xi \in \Xi} M^2_{2,\xi}(\x_i)}+ M k_n^{-1}    \lr{n^{-1} \sum_{i=1}^n\sup_{\xi \in \Xi} N^2_{2,\xi}(\x_i)}.
\end{align*}
Using $\alpha>2$ and $\delta>0$, and the strong law of large numbers which holds under \ref{hyp:weak:one}, the rhs tends to $0$ $\PP_\star-a.s.$ This implies \eqref{eq:convUnifB} and the proof is completed. 
\end{proof}
\begin{remark}
    Actually the proof of \eqref{eq:holder} is not immediate since one term is scalar and the other one is a matrix. To be specific, let $U$ be random variable  and $V$ be a real valued (possibly non squared) matrix. Denote $\|\cdot\|$ any norm on the matrices. Then consider the random variables $X,Y$ (they are therefore real valued) defined by 
    $$
    X=\frac{|U|^\gamma}{\PE^{1/s}[|U|^{\gamma s}]} \quad \mbox{and} \quad Y=\frac{\|V\|^\gamma}{\PE^{1/t}[\|V\|^{\gamma t}]}.
    $$
    Then, 
    $$ 
    \frac{\|U V\|^\gamma}{\PE^{1/s}[|U|^{\gamma s}] \times \PE^{1/t}[\|V\|^{\gamma t}]}=XY \leq \frac 1 s X^s + \frac 1 t Y^t = \frac 1 s \frac{|U|^{\gamma s}}{\PE[|U|^{\gamma s}]} + \frac 1 t \frac{\|V\|^{\gamma t}}{\PE[\|V\|^{\gamma t}]}.
    $$
    Taking the expectation and raising to the power $1/\gamma$ yields: 
    $$ 
     \frac{ \PE^{1/\gamma} \lrb{\|U V\|^\gamma}}{\PE^{1/\gamma s}[|U|^{\gamma s}] \times \PE^{1/\gamma t}[\|V\|^{\gamma t}]} \leq \frac 1 s + \frac 1 t =1,
    $$
    which completes the proof. 
\end{remark}

\section{Technical results}

\begin{lemma} \label{lem:consist:log}
    Let $\seq{U_k}{k\in\nset}$ be sequence of iid positive random variables on a common probability space $(\Omega,\mcf,\PP)$. Denote by $\PE$ the associated expectation operator. Assume that $\PE[U_1]=1$ and $\PE[|\log U_1|]<\infty$. Then, 
    $$ 
    \lim_{m \to \infty} \PE\lrb{\log \lr{m^{-1} \sum_{\ell=1}^m  U_k }}=0\,,\quad \PP-\as
    $$
\end{lemma}

\begin{proof}
    Setting 
    \begin{equation*}
        A_m \eqdef m^{-1} \sum_{\ell=1}^m \log U_k, \quad B_m \eqdef \log \lr{m^{-1} \sum_{\ell=1}^m  U_k }\,,\quad C_m \eqdef m^{-1} \sum_{\ell=1}^m U_k-1\eqsp,
    \end{equation*}   
    we have $A_m\leq B_m \leq C_m$ where the first inequality follows from Jensen's inequality and the second one from $\log(u) \leq u-1$. Since $\PE[|\log U_1|]<\infty$ and $\PE[|U_1|]<\infty$, the strong Law of Large Numbers applies to $\seq{A_m}{m\in \nset}$ and $\seq{B_m}{m\in \nset}$ and hence, 
\begin{align*}
    & \lim_{m \to \infty}\PE[A_m]=\PE[ \lim_{m \to \infty} A_m],\\
    & \lim_{m \to \infty}\PE[C_m]=\PE[ \lim_{m \to \infty} C_m].
\end{align*}
The general dominated convergence theorem then yields
$$ 
\lim_{m \to \infty}\PE[B_m]=\PE[ \lim_{m \to \infty} B_m] =\PE\lrb{\log 1 }=0.
$$
\end{proof}

\begin{remark}
    The general dominated convergence theorem can be found in Supplement C.2, Lemma 15 of \cite{daudelDoucPortier2020} - where it is stated with different upper and lower bounds having a common limit - or in Theorem 19 page 89 of \cite{roydenRA}, with symmetric bounds.
\end{remark}

In what follows, for any $s>0$ and any $\rset^\j$-valued random vector $V$ on a probability space $(\Omega,\mcf,\PP)$, with associated expectation operator $\PE$, we write $\norms{W}{s}=\lr{\PE[|W|^s]}^{1/s}$ where $|\cdot|$ is any norm on $\rset^\j$. 
\begin{lemma}[Marcinkiewicz-Zygmund inequality]
\label{lem:marcinZygm}
    For any $s\geq 2$, there exists a constant $M_s$ such that for any sequence of i.i.d. $\rset^\j$-valued random vectors $\seq{V_k}{k\in\nset}$ defined on a common probability space $(\Omega,\mcf,\PP)$, we have  
$$ 
 \norms{\lo^{-1}\sum_{\ell=1}^\lo V_i-\PE[V_i]}{s} \leq M_s \lo^{-1/2} \norms{V-\PE[V]}{s} \leq 2 M_s \lo^{-1/2} \norms{V}{s}.
$$
\end{lemma}

\begin{remark}
    \label{rem:marcinZyg}
    As $s \mapsto \norms{V}{s}$ is non-decreasing, \Cref{lem:marcinZygm} implies that for any $s>0$, 
    $$ 
     \norms{\lo^{-1}\sum_{\ell=1}^\lo V_i-\PE[V_i]}{s} \leq M_{s\vee 2} \lo^{-1/2} \norms{V-\PE[V]}{s\vee 2}.
    $$        
\end{remark}
\begin{lemma}
    \label{lem:log}
    For every $r>0$, we have
    \begin{align}
        \left|\log r-(r-1)\right| &\leq |(r-1)\log r|, \label{eq:bound:log:one}\\
        \left|\log r-\lrb{(r-1)-\frac {(r-1)^2} 2 }\right| &\leq (r-1)^2|\log r|.\label{eq:bound:log:two}
    \end{align}
\end{lemma}
\begin{proof}
    Let $r>0$ and set $f(t)=\log(1+t(r-1))$. Since  $f(1)-\lrb{f(0)+f'(0)}=\int_0^1 [f'(t)-f'(0)]\rmd t $,
    we deduce by straightforward algebra that
    \begin{align} \label{eq:log:zero}
        \log r-(r-1)=A(r),\quad \mbox{where} \quad A(r)\eqdef(r-1)^2 \int_0^1 \frac {t} {1+t(r-1)} \rmd t.
    \end{align}
    Taking the absolute value of $A (r)$ and using $0\leq t\leq 1$ to bound the inner term in the integral in
    \eqref{eq:log:zero}, we obtain $|A(r)| \leq  |(r-1) \log r|$. Similarly, 
    $$ 
    f(1)-\lrb{f(0)+f'(0)+\frac {f''(0)} 2}=\int_0^1 [f'(t)-f'(0)-tf''(0)]\rmd t.
    $$
Hence, 
    \begin{align}
        \label{eq:log}
        \log r-\lrb{(r-1)-\frac {(r-1)^2} 2 }=B(r)\quad \mbox{where} \quad B(r)\eqdef(r-1)^3 \int_0^1 \frac {t^2} {1+t(r-1)} \rmd t.
    \end{align}
    Taking the absolute value of $B (r)$ and using $t^2\leq 1$ to bound the inner term in the integral in
    \eqref{eq:log}, we obtain $ |B(r)| \leq (r-1)^2 |\log r|$ which concludes the proof. 
\end{proof}

\begin{lemma}
    \label{lem:borne:log:U}
    Let $\seq{U_k}{k\in\nset}$ be a sequence of i.i.d. $\rset^+_*$-valued random variables and defined on a common probability space $(\Omega,\mcf,\PP)$. Assume that $\PE[U_1]=1$ and define 
    $$ 
    \bar S_\lo= \lo^{-1}\sum_{\ell=1}^\lo U_k\,,\quad \lo\in\nset_*\,, 
    $$
Then, there exists a constant $M$ such that for any $\lo$,  
\begin{enumerate}[(i)]
        \item $\norms{\log \bar S_\lo}{4}  \leq \norms{\log U_1}{4}+4$, \label{eq:bound:logU:one}\\
        \item $\norms{\log \bar S_\lo}{2}  \leq M \lo^{-1/2} \norms{U_1}{4} \lr{\norms{\log U_1}{4}+1}$. \label{eq:bound:logU:two}
\end{enumerate} 
\end{lemma}
\begin{proof} We start with \ref{eq:bound:logU:one}. 
    Applying Jensen's inequality to the log function and for any
    positive $u$,  using the
    bounds, $\log u= 4 \log( u^{1/4} ) \leq 4(u^{1/4} -1) \leq 4 u^{1/4} $, we obtain: 
$$ 
 \lo^{-1} \sum_{\ell=1}^\lo \log U_k \leq \log(\bar S_\lo) \leq 4 \bar S_\lo^{1/4}.
$$
Hence, 
$$ 
 |\log \bar S_\lo| \leq \lo^{-1} \sum_{\ell=1}^\lo |\log U_k| + 4\bar S_\lo^{1/4}.
$$
Applying the triangular inequality for $\norms{\cdot}{4}$ and noting
that $\PE[S_\lo]=1$, we get  
$$ 
\norms{\log \bar S_\lo}{4} \leq \norms{\log U_1}{4} +4 \norms{S_\lo^{1/4}}{4}=\norms{\log U_1}{4} +4,
$$
which proves \ref{eq:bound:logU:one}. 

We now turn to \ref{eq:bound:logU:two}. According to \Cref{lem:log}, 
$$ 
 |\log \bar S_\lo| \leq |\bar S_\lo -1|+ |\bar S_\lo -1| \ |\log \bar S_\lo| .
$$
Using the triangular inequality for $\norms{\cdot}{2}$ and then the Cauchy-Schwarz inequality, 
$$ 
 \norms{\log \bar S_\lo}{2} \leq \norms{\bar S_\lo -1}{2}+ \norms{\bar S_\lo -1}{4} \ \norms{\log \bar S_\lo}{4} .
$$
Applying \ref{eq:bound:logU:one} to bound the last term and
Marcinkiewicz-Zygmund's bound to $\norms{\bar S_\lo -1}{k}$, $k \in
\{2,4\}$ and noting that $\norms{U_1}{2}\leq \norms{U_1}{4}$, we obtain \ref{eq:bound:logU:two}.   
\end{proof}

\begin{lemma}
    \label{lem:borne:autonorm}
    Let $\seq{(U_k,V_k)}{k\in\nset}$ be a sequence of i.i.d. $\rset^+_*\times \rset^\j$-valued random vectors and defined on a common probability space $(\Omega,\mcf,\PP)$. Assume that $\PE[U_1]=1$ and $\PE[V_1]=0$. Then, defining 
    $$ 
     W_\lo= \frac {\sum_{\ell=1}^\lo V_k} {\sum_{\ell=1}^\lo U_k}\,,\quad \lo\in\nset_*\,, 
    $$
    the following bounds hold :
    \begin{enumerate}[(i)]
    \item \label{item:autonorm:one} For every $\alpha>1$ and $\lo\geq 1$, $\norms{W_\lo}{\alpha} \leq \lo^{1/\alpha} \norms{V_1/U_1}{\alpha}$. 
    \item \label{item:autonorm:two} For every $\beta>0$ and $\alpha \geq 2\vee \beta$, there exists a constant $M$ such that for every $\lo\geq 1$ 
    $$ 
    \norms{W_\lo}{\beta} \leq M \lo^{-1/2+1/\alpha} \lrb{\lr{1+\norms{U_1}{\frac {\alpha \beta} {\alpha - \beta}\vee 2}} \norms{V_1/U_1}{\alpha}+ \norms{V_1}{\beta\vee 2}}.
    $$ 
    \item \label{item:autonorm:three} For every $\alpha,\beta >0$ such that $\alpha^{-1}+\beta^{-1}=1$, 
$$ 
 |\PE[W_\lo]| \leq \norms{1-\lo^{-1}\sum_{\ell=1}^\lo U_k}{\alpha} \norms{W_\lo}{\beta}.
$$
    \end{enumerate}
\end{lemma}
\begin{proof}
\begin{enumerate}[(i)]
    \item 
    By the triangular inequality for the norm $|\cdot|$, 
    $$ 
    |W_\lo| \leq \sum_{\ell=1}^\lo \lr{\frac{U_k} {\sum_{\j=1}^\lo U_\j}} \left| \frac {V_k}{U_k}\right| .
    $$
    Raising to the power $\alpha$ and applying Jensen's inequality to the function $u\mapsto u^\alpha$ yields  
    $$ 
     |W_\lo|^\alpha  \leq \sum_{\ell=1}^\lo \lr{\frac{U_k} {\sum_{\j=1}^\lo U_\j}} \left| \frac {V_k}{U_k}\right|^\alpha \leq \sum_{\ell=1}^\lo \left| \frac {V_k}{U_k}\right|^\alpha.
    $$
    Taking the expectation and raising to the power $1/\alpha$ completes the proof. 
    \item Set $v=\alpha/\beta \geq 1$ and let $u=\alpha /(\alpha-\beta)$ so that $u^{-1}+v^{-1}=1$. 
Since by straightforward algebra, 
\begin{equation}
    \label{eq:decomp:biais}
    W_\lo=\lr{1-\lo^{-1}\sum_{\ell=1}^\lo U_k} W_\lo + \lo^{-1}\sum_{\ell=1}^\lo V_k ,
\end{equation}
we have by Holder's inequality
\begin{align*}
    \norms{W_\lo}{\beta}&=\norms{\lr{1-\lo^{-1}\sum_{\ell=1}^\lo U_k} W_\lo}{\beta} + \norms{\lo^{-1}\sum_{\ell=1}^\lo V_k}{\beta}\\
    & \leq \norms{1-\lo^{-1}\sum_{\ell=1}^\lo U_k}{\beta u}  \norms{W_\lo}{\beta v} + \norms{\lo^{-1}\sum_{\ell=1}^\lo V_k}{\beta}.
\end{align*}
Using \Cref{rem:marcinZyg} to bound the first and third term of the rhs and using the bound (i) for the second term complete the proof. 
\item Applying the decomposition \eqref{eq:decomp:biais} and considering that $\PE[V_1]=0$, we obtain 
$$ 
 \PE[W_\lo] =\PE\lrb{\lr{1-\lo^{-1}\sum_{\ell=1}^\lo U_k} W_\lo}.
$$
Holder's inequality applied to the rhs finishes the proof. 
\end{enumerate}

\end{proof}

\section{Derivations for the unobserved heterogeneity model}

\vspace{0.3cm}

\noindent \textit{The notations used for the expectations in this appendix are informal and adopted for clarity.}

\vspace{0.3cm}

The ELBO is defined for a given dataset $(\x_i)_{1\leq i \leq n}$ by the function:
\[
\theta \mapsto \sum_{i=1}^n \mathbb{E}_{p(\bz)}\left[\log p_\theta(\x_i|\bz)\right] ,
\]
where $p(z)=e^{-z}e^{-e^{-z}}$ and $p_\theta(x|z) = \frac{1}{\sqrt{2\pi}} \exp\left(-\frac{(x - (\theta + z))^2}{2}\right)$.
Hence, $\tilde{\theta}_n^1$ minimizes:
\[
\theta \mapsto \mathbb{E}_{\hat{p}(x)} \mathbb{E}_{p(z)}\left[(x - (\theta + z))^2\right] ,
\]
where $\hat{p}(\cdot)=\frac{1}{n} \sum_{i=1}^n \delta_{\{\x_i\}}$ is the empirical distribution, which is minimized at 
$$
\tilde{\theta}_n^1 = \mathbb{E}_{\hat{p}(x)}\mathbb{E}_{p(z)}\left[x-z\right] = \frac{1}{n} \sum_{i=1}^n \x_i - \gamma 
$$
since $\mathbb{E}_{p(z)}\left[z\right]$ is equal to Euler's constant $\gamma$.

\vspace{0.2cm}

To get the expected ELBO maximizer $\tilde{\theta}_\infty^1$, we simply need to replace the empirical expectation by the true one, which leads to:
$$
\tilde{\theta}_\infty^1 = \mathbb{E}_{p_{\theta^*}(x)}\left[x\right] - \gamma = \theta^* 
$$
since
\begin{align*}
\mathbb{E}_{p_{\theta^*}(x)}\left[x\right] & = \int_x x p_{\theta^*}(x) dx \\
& = \int_x x \left( \int_z p_{\theta^*}(x|z) p(z) dz \right) dx \\
& = \int_z \left( \int_x x p_{\theta^*}(x|z) dx \right) p(z) dz \\
& = \int_z \mathbb{E}_{p_{\theta^*}(x|z)}\left[ x \right] p(z) dz \\
& = \int_z \left(\theta^*+z\right) p(z) dz \\
& = \theta^* + \int_z z p(z) dz \\
& = \theta^* + \mathbb{E}_{p(z)}\left[ z \right]  \\
& = \theta^* + \gamma .
\end{align*}

If we use overlapping MSLE for $k=1$ with $\bz\sim p(z)$, we have $\hat{\theta}_n^1$ that minimizes:
$$
\theta \mapsto \frac{1}{n} \sum_{i=1}^n - \log p_\theta(\x_i|\bz) = \frac{1}{n} \sum_{i=1}^n (\x_i - (\theta + \bz))^2 + \text{cst} = \left( \theta - \left\{\frac{1}{n} \sum_{i=1}^n \x_i - \bz \right\} \right)^2 + \text{cst'} 
$$
where $\text{cst}$ and $\text{cst'}$ denote constants independent of $\theta$. So 
$$ 
\hat{\theta}_n^1 = \frac{1}{n} \sum_{i=1}^n \x_i - \bz 
$$
and the expectation w.r.t.\ $\bz$ is
$$
\mathbb{E}_{p(z)}\left[\hat{\theta}_n^1\right] =\frac{1}{n} \sum_{i=1}^n \x_i - \gamma .
$$

If we use independent MSLE for $k=1$ with i.i.d.\ $(\bz_i)_{1\leq i \leq n}\sim p(z)$, we have $\hat{\theta}_n^1$ that minimizes:
$$
\theta \mapsto \frac{1}{n} \sum_{i=1}^n - \log p_\theta(\x_i|\bz_i) = \frac{1}{n} \sum_{i=1}^n (\x_i - (\theta + \bz_i))^2 + \text{cst} = \left( \theta - \frac{1}{n} \sum_{i=1}^n (\x_i - \bz_i) \right)^2 + \text{cst'} 
$$
where $\text{cst}$ and $\text{cst'}$ denote once again constants independent of $\theta$. So 
$$ 
\hat{\theta}_n^1 = \frac{1}{n} \sum_{i=1}^n (\x_i - \bz_i) 
$$
and the expectation w.r.t.\ $\bz_i$'s is
$$
\mathbb{E}_{p(z_1),...,p(z_n)}\left[\hat{\theta}_n^1\right] = \frac{1}{n} \sum_{i=1}^n x_i - \gamma .
$$

\end{appendix}

\vspace{1cm}
\bibliography{biblio}      

\begin{thebibliography}{}

\bibitem[Blei et~al., 2017]{blei2017variational}
Blei, D.~M., Kucukelbir, A., and McAuliffe, J.~D. (2017).
\newblock Variational inference: A review for statisticians.
\newblock {\em Journal of the American Statistical Association}, 112(518):859--877.

\bibitem[Bornschein and Bengio, 2015]{bornschein2014reweighted}
Bornschein, J. and Bengio, Y. (2015).
\newblock Reweighted wake-sleep.
\newblock {\em International Conference on Learning Representations}.

\bibitem[Burda et~al., 2016]{IWAE2015}
Burda, Y., Grosse, R., and Salakhutdinov, R. (2016).
\newblock Importance weighted autoencoders.
\newblock {\em International Conference on Learning Representations}.

\bibitem[Cameron and Trivedi, 2005]{cameron2005microeconometrics}
Cameron, A.~C. and Trivedi, P.~K. (2005).
\newblock Microeconometrics: Methods and applications.
\newblock {\em Cambridge University Press}.

\bibitem[Daudel et~al., 2023]{daudel2023alpha}
Daudel, K., Benton, J., Shi, Y., and Doucet, A. (2023).
\newblock Alpha-divergence variational inference meets importance weighted auto-encoders: Methodology and asymptotics.
\newblock {\em Journal of Machine Learning Research}, 24(243):1--83.

\bibitem[Daudel et~al., 2021]{daudelDoucPortier2020}
Daudel, K., Douc, R., and Portier, F. (2021).
\newblock Infinite-dimensional gradient-based descent for alpha-divergence minimisation.
\newblock {\em Annals of Statistics}, 49(4):2250--2270.

\bibitem[Dhekane, 2021]{dhekane2021improving}
Dhekane, E.~G. (2021).
\newblock On improving variational inference with low-variance multi-sample estimators.
\newblock {\em MSc Thesis, University of Montreal}.

\bibitem[Domke and Sheldon, 2018]{domke2018importance}
Domke, J. and Sheldon, D.~R. (2018).
\newblock Importance weighting and variational inference.
\newblock {\em Advances in Neural Information Processing Systems}, 31.

\bibitem[Domke and Sheldon, 2019]{domke2019divide}
Domke, J. and Sheldon, D.~R. (2019).
\newblock Divide and couple: Using monte carlo variational objectives for posterior approximation.
\newblock {\em Advances in Neural Information Processing Systems}, 32.

\bibitem[Ferguson, 1996]{Ferguson1996}
Ferguson, T. (1996).
\newblock {\em A Course in Large Sample Theory}.
\newblock Routledge.

\bibitem[Geyer, 1994]{geyer1994convergence}
Geyer, C.~J. (1994).
\newblock On the convergence of {M}onte {C}arlo maximum likelihood calculations.
\newblock {\em Journal of the Royal Statistical Society: Series B (Methodological)}, 56(1):261--274.

\bibitem[Gourieroux et~al., 1996]{gourieroux1996simulation}
Gourieroux, C., Gourieroux, M., Monfort, A., and Monfort, D.~A. (1996).
\newblock {\em Simulation-based econometric methods}.
\newblock Oxford university press.

\bibitem[Huang and Courville, 2019]{huang2019note}
Huang, C.-W. and Courville, A. (2019).
\newblock Note on the bias and variance of variational inference.
\newblock {\em arXiv preprint arXiv:1906.03708}.

\bibitem[Huang et~al., 2019]{huang2019hierarchical}
Huang, C.-W., Sankaran, K., Dhekane, E., Lacoste, A., and Courville, A. (2019).
\newblock Hierarchical importance weighted autoencoders.
\newblock In {\em International Conference on Machine Learning}, pages 2869--2878.

\bibitem[Ipsen et~al., 2020]{ipsen2020not}
Ipsen, N.~B., Mattei, P.-A., and Frellsen, J. (2020).
\newblock not-miwae: Deep generative modelling with missing not at random data.
\newblock {\em arXiv preprint arXiv:2006.12871}.

\bibitem[Jordan et~al., 1999]{Jordan1999}
Jordan, M.~I., Ghahramani, Z., Jaakkola, T.~S., and Saul, L.~K. (1999).
\newblock An introduction to variational methods for graphical models.
\newblock {\em Machine Learning}, 37:183--233.

\bibitem[Kingma and Welling, 2014]{kingma2013auto}
Kingma, D.~P. and Welling, M. (2014).
\newblock Auto-encoding variational bayes.
\newblock {\em International Conference on Learning Representations}.

\bibitem[Klys et~al., 2018]{klys2018joint}
Klys, J., Bettencourt, J., and Duvenaud, D. (2018).
\newblock Joint importance sampling for variational inference.
\newblock {\em International Conference on Learning Representation Workshop Track}.

\bibitem[Le et~al., 2018]{le2018auto}
Le, T.~A., Igl, M., Rainforth, T., Jin, T., and Wood, F. (2018).
\newblock Auto-encoding sequential {M}onte {C}arlo.
\newblock {\em International Conference on Learning Representations}.

\bibitem[Lee, 1992]{lee1992efficiency}
Lee, L.-F. (1992).
\newblock On efficiency of methods of simulated moments and maximum simulated likelihood estimation of discrete response models.
\newblock {\em Econometric Theory}, 8(4):518--552.

\bibitem[Lee, 1995]{lee1995asymptotic}
Lee, L.-F. (1995).
\newblock Asymptotic bias in simulated maximum likelihood estimation of discrete choice models.
\newblock {\em Econometric Theory}, 11(3):437--483.

\bibitem[Lerman and Manski, 1981]{lerman1981use}
Lerman, S. and Manski, C. (1981).
\newblock On the use of simulated frequencies to approximate choice probabilities.
\newblock {\em Structural Analysis of Discrete Data with Econometric Applications}, 10:305--319.

\bibitem[Li{\'e}vin et~al., 2020]{lievin2020optimal}
Li{\'e}vin, V., Dittadi, A., Christensen, A., and Winther, O. (2020).
\newblock Optimal variance control of the score-function gradient estimator for importance-weighted bounds.
\newblock {\em Advances in Neural Information Processing Systems}, 33:16591--16602.

\bibitem[Maddison et~al., 2017]{maddison2017filtering}
Maddison, C.~J., Lawson, J., Tucker, G., Heess, N., Norouzi, M., Mnih, A., Doucet, A., and Teh, Y. (2017).
\newblock Filtering variational objectives.
\newblock {\em Advances in Neural Information Processing Systems}, 30.

\bibitem[Mattei and Frellsen, 2019]{mattei2019miwae}
Mattei, P.-A. and Frellsen, J. (2019).
\newblock Miwae: Deep generative modelling and imputation of incomplete data sets.
\newblock In {\em International Conference on Machine Learning}, pages 4413--4423. PMLR.

\bibitem[Mattei and Frellsen, 2022]{mattei2022uphill}
Mattei, P.-A. and Frellsen, J. (2022).
\newblock Uphill roads to variational tightness: Monotonicity and {M}onte {C}arlo objectives.
\newblock {\em arXiv preprint arXiv:2201.10989}.

\bibitem[Mayer et~al., 2020]{mayer2020missdeepcausal}
Mayer, I., Josse, J., Raimundo, F., and Vert, J.-P. (2020).
\newblock Missdeepcausal: Causal inference from incomplete data using deep latent variable models.
\newblock {\em arXiv preprint arXiv:2002.10837}.

\bibitem[McFadden, 1989]{mcfadden1989method}
McFadden, D. (1989).
\newblock A method of simulated moments for estimation of discrete response models without numerical integration.
\newblock {\em Econometrica: Journal of the Econometric Society}, pages 995--1026.

\bibitem[Naesseth et~al., 2018]{naesseth2018variational}
Naesseth, C., Linderman, S., Ranganath, R., and Blei, D. (2018).
\newblock Variational sequential {M}onte {C}arlo.
\newblock In {\em International conference on artificial intelligence and statistics}, pages 968--977. PMLR.

\bibitem[Nowozin, 2018]{nowozin2018debiasing}
Nowozin, S. (2018).
\newblock Debiasing evidence approximations: On importance-weighted autoencoders and jackknife variational inference.
\newblock In {\em International Conference on Learning Representations}.

\bibitem[Pakes, 1986]{pakes1986patents}
Pakes, A. (1986).
\newblock Patents as options: Some estimates of the value of holding european patent stocks.
\newblock Technical report, National Bureau of Economic Research.

\bibitem[Pakes and Pollard, 1989]{pakes1989simulation}
Pakes, A. and Pollard, D. (1989).
\newblock Simulation and the asymptotics of optimization estimators.
\newblock {\em Econometrica: Journal of the Econometric Society}, pages 1027--1057.

\bibitem[Rainforth et~al., 2018a]{rainforth2018nesting}
Rainforth, T., Cornish, R., Yang, H., Warrington, A., and Wood, F. (2018a).
\newblock On nesting {M}onte {C}arlo estimators.
\newblock In {\em International Conference on Machine Learning}, pages 4267--4276. PMLR.

\bibitem[Rainforth et~al., 2018b]{rainforth18b}
Rainforth, T., Kosiorek, A., Le, T.~A., Maddison, C., Igl, M., Wood, F., and Teh, Y.~W. (2018b).
\newblock Tighter variational bounds are not necessarily better.
\newblock In Dy, J. and Krause, A., editors, {\em Proceedings of the 35th International Conference on Machine Learning}, volume~80 of {\em Proceedings of Machine Learning Research}, pages 4277--4285. PMLR.

\bibitem[Rezende et~al., 2014]{rezende2014stochastic}
Rezende, D.~J., Mohamed, S., and Wierstra, D. (2014).
\newblock Stochastic backpropagation and approximate inference in deep generative models.
\newblock In {\em International conference on machine learning}, pages 1278--1286. PMLR.

\bibitem[Royden and Fitzpatrick, 2015]{roydenRA}
Royden, H. and Fitzpatrick, P. (2015).
\newblock {\em Real Analysis}.
\newblock 4th edition.

\bibitem[Sung and Geyer, 2007]{sung2007monte}
Sung, Y.~J. and Geyer, C.~J. (2007).
\newblock Monte {C}arlo likelihood inference for missing data models.
\newblock {\em The Annals of Statistics}, 35(3):990--1011.

\bibitem[Train, 2003]{train2003discrete}
Train, K.~E. (2003).
\newblock {\em Discrete Choice Methods with Simulation}.
\newblock Cambridge University Press.

\bibitem[Tucker et~al., 2019]{tucker2018doubly}
Tucker, G., Lawson, D., Gu, S., and Maddison, C.~J. (2019).
\newblock Doubly reparameterized gradient estimators for {M}onte {C}arlo objectives.
\newblock {\em International Conference on Learning Representations}.

\bibitem[Van~der Vaart, 1998]{VanDerVaart1999}
Van~der Vaart, A.~W. (1998).
\newblock {\em Asymptotic Statistics}.
\newblock Cambridge Series in Statistical and Probabilistic Mathematics. Cambridge University Press.

\end{thebibliography}
\end{document}